\documentclass[11pt,a4paper]{amsart}
\usepackage[centering, margin=1in]{geometry}
\usepackage[latin1]{inputenc}
\usepackage[english]{babel}
\usepackage[T1]{fontenc}
\usepackage{amsmath}
\usepackage{amssymb}
\usepackage{amsthm}
\usepackage{array,booktabs}
\usepackage{graphicx}
\usepackage[OT2,T1]{fontenc}
\DeclareSymbolFont{cyrletters}{OT2}{wncyr}{m}{n}
\DeclareMathSymbol{\Sha}{\mathalpha}{cyrletters}{"58}
\usepackage{hyperref}
\usepackage{euflag}
\usepackage{lmodern}
\usepackage{tcolorbox}
\usepackage{quoting}
\usepackage{enumerate}
\usepackage{mathtools}
\usepackage{csquotes}
\usepackage{stmaryrd}
\usepackage{cite}
\usepackage{tikz}
\usepackage{todonotes}
\usepackage{tikz-cd}
\usetikzlibrary{matrix,arrows,decorations.pathmorphing}
\usepackage{wrapfig}
\mathchardef\mhyphen="2D
\renewcommand{\Re}{\operatorname{Re}}

\def\R{\ensuremath\mathbb{R}}
\def\C{\ensuremath\mathbb{C}}

\def\Z{\ensuremath\mathbb{Z}}
\def\Q{\ensuremath\mathbb{Q}}
\def\Oo{\ensuremath\mathcal{O}}
\def\N{\ensuremath\mathbb{N}}

\def\F{\ensuremath\mathbb{F}}

\def\Frob{\ensuremath\mathrm{Frob}}

\usepackage{hyperref}
\usepackage[final]{pdfpages}
\usepackage{verbatim}
\newtheorem{theorem}{Theorem}[section]
\newtheorem{definition}[theorem]{Definition}
\newtheorem{corollary}[theorem]{Corollary}
\newtheorem{lemma}[theorem]{Lemma}

\newtheorem{proposition}[theorem]{Proposition}

\newtheorem*{proposition*}{Proposition}
\theoremstyle{remark}
\newtheorem{remark}{Remark}[section]
\newtheorem{example}{Example}[section]

\def\eps{\ensuremath\varepsilon}
\def\rank{\text{\rm rank}}

\def\im {\text{\rm im}}

\def\Gal{\text{\rm Gal}}

\def\SL{\mathrm{SL}}

\def\GL{\mathrm{GL}}

\def\modulo{\text{ \rm mod }}

\def\1{\mathbf{1}}
\def\cond{\text{\rm cons}}
\def\ord{\text{\rm ord}}

\DeclareMathOperator{\disc}{disc}

\DeclareMathOperator{\Kl}{Kl}

\DeclareFontFamily{U}{wncy}{}
    \DeclareFontShape{U}{wncy}{m}{n}{<->wncyr10}{}
    \DeclareSymbolFont{mcy}{U}{wncy}{m}{n}
    \DeclareMathSymbol{\Sh}{\mathord}{mcy}{"58} 
\def\modulo{\text{ \rm mod }}

\def\cond{\text{\rm cond}}
\def\ord{\text{\rm ord}}
\def\pmod{\text{ \rm mod }}
\numberwithin{equation}{section}

\numberwithin{equation}{section}

\begin{document}
\title{Diophantine rank stability and non-vanishing of $L$-functions}
\author{Marius Fischer and Asbj\o rn Christian Nordentoft}

\address{Department of Mathematics, Aarhus University, 1530-432, 8000 Aarhus C, Denmark}

\email{\href{mailto:marius.fischer@math.au.dk}{marius.fischer@math.au.dk}}

\address{University of  Copenhagen, Universitetsparken 5, 2100
Copenhagen \O, Denmark}

\email{\href{mailto:nordentoft@math.ku.dk}{nordentoft@math.ku.dk}}

\date{\today}

\thanks{We are grateful to David Burns, Ian Kiming and Paul Nelson for useful conversations. We also thank Daniel Kriz, Anwesh Ray and Ariel Weiss for their helpful comments on earlier drafts of this paper. The first author is supported by grant VIL54509 from Villum Fonden. Suggestions from artificial intelligence (Gemini 3.1 Pro) inspired the counting argument in the proof of Theorem \ref{thm:non-van} and the proof of Lemma \ref{lem:orderly-unipotent}.}




\begin{abstract}
Let $A/\mathbb{Q}$ be a modular abelian variety of analytic rank $0$. If $G$ is a non-trivial finite abelian group such that all prime factors of $\lvert G \rvert$ are sufficiently large in terms of $A$, we show that there are infinitely many $G$-extensions $F/\mathbb{Q}$ such that $A(F)$ is finite. When $A$ is a rational elliptic curve of analytic rank zero with no exceptional primes, or the product of two such curves, the same conclusion holds without any assumptions on $|G|$. Our proof relies on new simultaneous non-vanishing results for twisted central $L$-values of even-weight holomorphic newforms. These results are obtained via novel constructions related to horizontal $p$-adic $L$-functions and are of independent interest.
\end{abstract}
\maketitle
\section{Introduction}\label{arithmetic-applications}
An abelian variety $A$  over $\Q$ is \emph{Diophantine stable} for a field extension $F/\Q$ if $A$ does not acquire new points over $F$, i.e.\ $A(F)=A(\Q)$. The study of \emph{Diophantine stability} for abelian varieties was initiated by Mazur and Rubin \cite{mazur-rubin-larsen2018} and asks to find (many) number fields of a certain kind (e.g.\ fixed degree, fixed Galois group) over which the abelian variety does not acquire new points. This paper is concerned with \emph{Diophantine rank stability} which asks for the weaker condition $\rank_\Z A(F)=\rank_\Z A(\Q)$. Gaining quantitative understanding on the frequency of when this happens question is a central topic in arithmetic statistics \cite{Golfeld5050,Smith1,Smith2}, for example through the congruent number problem \cite{ChaoLi20}, and rank stability has also played an important role in the recent resolution of Hilbert's tenth problem over number fields \cite{AlpBhargaveWeiShnidman,KoymansPaganoHilbert10}. Our approach to this problem is via non-vanishing of $L$-functions, and the results we obtain apply to general newforms of even weight, vastly generalizing the results of Fearnley, Kisilevsky and Kuwata \cite{FeKiKu} on a conjecture of David, Fearnley and Kisilevsky \cite{DaFeKi07}, see Section \ref{sec:nonvanishing} below. 
 
Recall that  an abelian variety $A/\Q$ is of \emph{$\GL_2$-type} if there  exists a $\Q$-algebra embedding $K\subset \mathrm{End}_\Q(A)\otimes_\Z \Q$ where $K$ is a field of degree equal to the dimension of $A$. It follows from the results in \cite{Ribet04} and \cite{KhareWintenberger2009a,KhareWintenberger2009b} that an abelian variety $A/\Q$ is isogenous to a product of $\GL_2$-type abelian varieties exactly if $A$ is \emph{modular}, meaning there exists a surjective morphism $J_1(N)\rightarrow A$ from the Jacobian of the modular curve $X_1(N)$. A key example is when $A=E_1\times \cdots \times E_n$ is a product of rational elliptic curves. The (Hasse--Weil) $L$-function  $L(A,s)$ of such an $A/\Q$ satisfies analytic continuation, and we define the \emph{analytic rank} as the order of vanishing at the central point $s=1$: 
$$r_\mathrm{an}(A):=\ord_{s=1} L(A,s).$$
We are interested in the case where $r_\mathrm{an}(A)=0$. It then follows from a result of Kato \cite{Kato} (see Corollary \ref{cor:Kato} below) that $A(\Q)$ is finite. The principal goal of this paper is to construct many abelian extension $F/\Q$ with specified Galois group for which $A(F)$ remains finite. To quantify this  we define for $G$ a finite abelian group and $X\geq 1$  the following family of fields:
\begin{equation}
    \mathcal{F}_{G}(X):=\{  F/\Q \text{ \rm Galois}\mid \Gal(F/\Q)\cong G, |\disc(F)|\leq X \}.
\end{equation} 
It is a result of Wright \cite{Wright} that 
\begin{equation}
    \label{eq:wright} |\mathcal{F}_G(X)|=(c(G)+o(1))X^{a(G)}(\log X)^{b(G)},\quad \text{as }X\rightarrow \infty,
\end{equation} 
for a positive constant $c(G)>0$, and
\begin{equation}
    a(G):=\frac{p}{|G|(p-1)},\quad b(G):= \frac{p^m-1}{p-1}, 
\end{equation}
where $p$ is the smallest prime number dividing $|G|$, and $p^m$ is the largest $p$-power order cyclic factor of $G$. 
We obtain the following rank stability result.
    \begin{theorem}\label{thm:abelianvar}
    Let $A/\Q$ be a modular abelian variety with analytic rank $r_\mathrm{an}(A)=0$. Let $G$ be a finite abelian group such that any prime divisor of $|G|$ is  sufficiently large (depending on $A$). Then there exists $\kappa<1$ such that  
    \begin{align}\label{eq:conclusioncorabelian}
|\{ F\in \mathcal{F}_G(X)\mid A(F)\text{ \rm is finite} \}| \gg \frac{X^{a(G)}}{(\log X)^{\kappa}},\quad X\rightarrow \infty.
\end{align}  
Furthermore, if $A=E_1\times \cdots \times E_n$ is a product of rational elliptic curves of analytic rank $0$ then it suffices that all prime divisors of $|G|$ are at least $\max\{167,3n+2\}$. If, in addition, $E_1,\ldots,E _n$ are all semi-stable, the lower bound $\max\{11,n+1\}$ is sufficient.
\end{theorem}
In the case where $A$ is a product of two elliptic curves we have better control on the ``sufficiently large'' part, and, as a consequence, we can resolve the existence of Diophantine rank stability for $G$-extensions for \emph{all} finite abelian groups $G$ in many situations. 
\begin{corollary}\label{cor:E1E2}
    Let $E_1,E_2/\Q$ be two elliptic curves of analytic rank $0$ for which $E_1[p]$ and $E_2[p]$ are irreducible $\Gal(\overline{\Q}/\Q)$-representations for all primes $p$. Let $G$ be a finite abelian group. Then there exists $\kappa<1$ such that 
    \begin{align}\label{eq:conclusioncor2}
|\{  F\in \mathcal{F}_G(X)\mid   E_1(F),E_2(F)\text{ \rm are finite} \}| \gg \frac{X^{a(G)}}{(\log X)^{\kappa}},\quad X\rightarrow \infty.
\end{align}  
\end{corollary}
\noindent
When ordered by height, a positive proportion of all rational elliptic curves have analytic rank $0$ and only irreducible mod $p$ representations. Indeed, it is a result of Duke \cite{duke1997} that 100 \% of all rational elliptic curves have no exceptional primes, i.e. all mod $p$ representations are surjective. Moreover, by work of Skinner--Urban \cite{skinner-urban2014} and Bhargava--Shankar \cite{bhargava-shankar2015}, a positive proportion of rational elliptic curves have analytic rank $0$. A concrete example of a pair $(E_1,E_2)$ that satisfies the conditions in Corollary \ref{cor:E1E2} is
\begin{equation*}
    E_1\,:\, y^2+xy=x^3-x^2-8x-7\quad\textrm{and}\quad E_2\,:\, y^2+xy=x^3+x^2-27x-67
\end{equation*}
with LMFDB-labels $109.a1$ and $106.b1$, respectively. Many more examples can be easily found in the LMFDB-database, for example by searching for elliptic curves of rank $0$ whose isogeny class has size $1$.

It is possible to relax the conditions on $E_1$ and $E_2$ in Corollary \ref{cor:E1E2}. For example, if the mod $p$ representation attached to $E_1$ is reducible for some prime $p\geq 5$, it suffices that $E_1$ has a prime of additive reduction different from $p$. In this way, we can also give examples of CM-elliptic curves for which the conclusion of Corollary \ref{cor:E1E2} still holds, see Section \ref{sec:ellipticcurves} for further details and concrete examples. We remark that Diophantine rank stability results for CM-elliptic curves seems to be rare.\\



\noindent
Theorem \ref{thm:abelianvar} and Corollary \ref{cor:E1E2} above are the first Diophantine rank stability results that apply to arbitrary abelian extensions. Prior to our work, Diophantine rank stability had mainly been studied over extensions of prime power degree \cite{mazur-rubin-larsen2018,RayWeston,ray-pathak26}, and in special classes of non-abelian extensions \cite{ShnidmanWeiss,PathakRay}. Furthermore, very little was known about Diophantine rank stability for non-simple varieties (e.g.\ when $A$ is a product of elliptic curves), as there are well-known complications that arise in this case \cite[Remark 10.4]{mazur-rubin-larsen2018}. The few existing results in this direction \cite[Theorem 1.1]{munshi2012} apply only to quadratic extensions or to cyclic extension, but under restrictive conditions\footnote{In \cite[Theorem 1.2]{KrizNordentoft23} simultaneous non-vanishing results were obtained by Kriz and the second named author under the assumption of the existence of ``joint orderly primes''. In this paper we show that this happens for $p$ sufficiently large, see Theorem \ref{thm:non-van}.} \cite[Corollary 1.3]{KrizNordentoft23}.


Most approaches to Diophantine rank stability follow \cite{mazur-rubin-larsen2018}, being arithmetic in nature, and proceed by bounding Selmer groups. Our approach takes an alternative route, namely via the study of non-vanishing of $L$-functions using $p$-adic methods (for a similar approach using archimedean methods/analytic number theory, see \cite[Theorem 1.6]{MV02}).   
\subsection{Non-vanishing of $L$-functions}
The most general approach to the non-vanishing problem seems to come from analytic number theory and consists in embedding the $L$-function in a family and calculating the average using analytic tools, see e.g. \cite{MurtyMurty91,BlFoKoMiMiSa18}. In the case where the central value is \emph{critical} in the sense of Deligne \cite{Deligne79}, a different route opens up by applying $p$-adic methods (including here mod $p$ methods). Interestingly, analytic tools are often needed as an input in the $p$-adic theory \cite{Rohrlich84,CornutVatsal07,RadziwillYang23}. However, there are examples where the algebraic methods yield non-vanishing results directly \cite{Merel09,DiJaRa20}. 
 A beautiful example of this approach is the work of Fearnley--Kisilevsky--Kuwata \cite{FeKiKu} concerned with the non-vanishing problem for characters of \emph{fixed} prime order motivated by the conjectures of  David--Fearnley--Kisilevsky \cite{DaFeKi07}. To quantify the problem and state their results, we put\footnote{Note that when $d$ is not a prime the structure of the family $\mathcal{F}_d(X)$ is quite different from  $\mathcal{F}_G(X)$ with $G=\Z/d$ since the latter sees characters of order \emph{dividing} $d$. The families $\mathcal{F}_d(X)$ are however more natural  from the point of view of analytic number theory \cite{ChoPark}.}
\begin{equation}\label{eq:Kd}\mathcal{K}_d(X):=\{ \chi\modulo D\mid \text{primitive of order $d$},D\leq X\},\quad X\geq 1,\end{equation}
which we note satisfies the asymptotic $| \mathcal{K}_d(X)|=(c_d+o(1)) (\log X)^{\sigma_0(d)-2}X$ as $X\rightarrow \infty$ for some constant $c_d>0$ 
where $\sigma_0(d)$ denotes the number of divisors of $d$ \cite[Corollary 5.8]{KrizNordentoft23}. 
Using mod $p$ methods, the authors prove in \cite{FeKiKu} that for an elliptic curve $E/\Q$ with analytic rank equal to $0$ (i.e.\ $L(E,1)\neq 0$) and a prime $p$ sufficiently large (depending on $E$), it holds that
\begin{align}\label{eq:conclusion}
|\{ \chi\in \mathcal{K}_p(X)\mid   L(E, \chi,1) \neq0\}| \gg \frac{X}{(\log X)^{\kappa}},\quad \text{as }X\rightarrow \infty,
\end{align}
for some $0<\kappa<1$ (depending on $E$ and $p$). Again by the work of Kato, this implies Theorem \ref{thm:abelianvar} in the case where $A$ is an elliptic curve $E/\Q$ and $G=\Z/p$ (and generalizes easily to the case $G=\Z/p^m)$. 
In \cite{KrizNordentoft23} the second named author together with Daniel Kriz obtained a wide range of generalizations of the above results via the study of \emph{horizontal $p$-adic $L$-functions}, which are certain integral $p$-adic measures  interpolating twists of $p$-power order and conductor prime to $p$. In this paper we obtain, in particular, new results on the existence of horizontal $p$-adic $L$-functions which imply the following generalization of the non-vanishing result (\ref{eq:conclusion}).

\begin{theorem}\label{thm:nonvanishing}
Let $f_1,\ldots, f_n$ be holomorphic newforms of even weights $k_1,\ldots, k_n$. Assume that $L(f_i,k_i/2)\neq0$ for all $i=1,\ldots, n$. Then there exists  $M\geq 1$ (depending on $f_1,\ldots, f_n$) such that if $d\geq 2$ is an integer with a prime divisor greater than or equal to $M$ then the following holds: there exists a constant $\kappa<1$ such that 
\begin{align}\label{eq:conclusionthm}
|\{ \chi\in \mathcal{K}_d(X)\mid   L(f_1, \chi,k_1/2)\cdots L(f_n, \chi,k_n/2) \neq0\}| \gg \frac{X}{(\log X)^{\kappa}},\quad \text{as }X\rightarrow \infty.
\end{align}
Furthermore, for $n=2$ the above conclusion holds as long as $d$ is divisible by a prime $p$ satisfying the following precise condition: for $i=1,2$ there exists a prime $\mathfrak{p}_i$ above $p$ in the Hecke field of $f_i$ such that the corresponding residual representation mod $\mathfrak{p}_i $ is absolutely irreducible.
\end{theorem}
\begin{remark}
If $f$ is a holomorphic newform of weight at least $2$, then the mod $p$ representations attached to $f$ are absolutely irreducible for all $p$ sufficiently large, see Lemma \ref{CM-case} and Lemma \ref{non-CM-case}. If $f_1,\ldots, f_n$ are all of weight $2$ corresponding to elliptic curves of analytic rank 0, then, as explained in Section \ref{sec:ellipticcurves}, one can take $M=\max\{167, 3n+2\}$  in Theorem \ref{thm:nonvanishing}, and $M=\max\{11,n+1\}$ if the curves are semi-stable. 
\end{remark}
\begin{remark}
    The above results are quite astonishing from the point of view of analytic number theory since calculating the first moment of the corresponding twist family is out of range of analytic techniques due to the appearance of (the unwieldy) higher order Gauss sums in the dual sum of the approximate functional equation. For the state of the art, see \cite{Sound00,XiannanLi2024} for results on quadratic characters and  \cite{BaierYoung,DaviddeFaveriETC24} for cubic characters.
\end{remark}
As alluded to earlier, the above theorem generalizes Theorem A in \cite{FeKiKu} which corresponds to the case $n=1$, $f_1=f_E$ is of weight 2 associated to an elliptic curve $E/\Q$, and $d=p$ is a sufficiently large prime. 
Finally, in case of a single elliptic curve $E$ and $d=p$ prime, the above theorem holds for all $p\geq 167$. In contrast, the methods of \cite{FeKiKu} rely on $p$ not dividing the algebraic part of the central $L$-value whose prime factors cannot be bounded independently of $E$. 
Comparing to \cite[Theorem 1.2]{KrizNordentoft23} the above theorem  applies to \emph{any} newforms of even weights with non-vanishing central $L$-values  (as opposed to under certain ``joint big image'' assumption, see \cite[Remark 1.1]{KrizNordentoft23}). Finally, the last part of the above theorem in the case $n=2$ generalizes the last bullet of \cite[Theorem 1.1]{KrizNordentoft23} from a single elliptic curve to a general pair of even weight holomorphic cusp forms.

\subsection{Proof strategy}
Our argument proceeds by introducing a number of refinements in the theory of horizontal $p$-adic $L$-functions developed in \cite{KrizNordentoft23}. We now explain the idea behind the proof of Theorem  \ref{thm:abelianvar}  in the case where $A=E_1\times\cdots\times E_k$ is a product of rational elliptic curves of analytic rank $0$. 

Suppose $G$ is a non-trivial finite abelian group. By Kato's explicit reciprocity law (or more precisely its consequences in Corollary \ref{cor:Kato}), it suffices to construct (many) abelian extensions $F/\Q$ with Galois group $G$ such that $ L(E_i,\chi,1)\neq0 $ for all $i=1,\ldots, k$ and all characters $\chi$ of $\Gal(F/\Q)$. In other words, the proof is reduced to a simultaneous non-vanishing problem for $L$-functions with simultaneity in two directions, namely the family of elliptic curves $E_1,\ldots,E_n$ the and characters $\chi$ of $\Gal(F/\Q)$. Consider now the  primary decomposition of $G$:
\begin{equation}
 \label{eq:primary}   G\cong \Z/p_1^{\alpha_1} \times\cdots \times \Z/p_t^{\alpha_t}.
\end{equation}
Our strategy is to construct the sought-after $F$ as the compositum of $\Z/p_i^{\alpha_i}$-extensions for each $i$ with pairwise coprime discriminants obtained recursively using the theory of \emph{horizontal $p$-adic $L$-functions}: For $p=p_1$ as in the primary decomposition (\ref{eq:primary}), we have (under certain assumptions) a $p$-adic measure associated to $A$: 
\begin{equation}
    \label{eq:intromeasure} \nu_{A,p}\in \Z_p\llbracket\prod_{n\in \N} \Z/p^{m_n}\rrbracket:=\varprojlim_{N} \Z_p[\prod_{n\leq N} \Z/p^{m_n}], 
\end{equation} 
where $m_n=v_p(\ell_n-1)$, and $\ell_n$ are so-called \emph{orderly primes} for the  mod $p$ Galois representation associated with $A$ (which we define in Section \ref{subsubsec:orderly} and whose existence is a key contribution of the present paper). This measure is characterized by the following interpolation property: For a  Dirichlet character $\chi$  of $p$-power order and conductor dividing $\prod_{n\in \N}\ell_n$, we can consider it as a $\C_p^\times$-valued continuous character of the profinite group $\prod_{n\in \N} \Z/p^{m_n}$, and it satisfies
$$\nu_{A,p}(\chi)= (..) L(A,\chi,1)/\Omega_A=(..) L(E_1,\chi,1)\cdots L(E_k,\chi,1)/(\Omega^+_{E_1}\cdots \Omega^+_{E_k}),$$
where $(..)$ denotes certain explicit non-vanishing Euler factors at the primes dividing the conductor of $\chi$, and $\Omega_A=\Omega^+_{E_1}\cdots \Omega^+_{E_k}\in \C^\times$ denotes a \emph{period} of $A$ so that the right-hand side is algebraic and $p$-integral (we can pick $\Omega^+_{E_i}$ to be the real N\'{e}ron period of $E_i$, see Theorem \ref{thm:padicL} for details). Note that by assumption of analytic rank $0$, the value at the trivial character $\nu_{A,p}(\mathbf{1})$ is non-zero, and so we can appeal to the general structure theorem for the set of non-vanishing characters obtained by Daniel Kriz and the second-named author in \cite{KrizNordentoft23}. This argument yields the proof of  Theorem \ref{thm:nonvanishing}. 

For the problem of Diophantine rank stability this does not suffice, and new ideas are needed: We require the stronger condition that $\nu_{A,p}(\chi^j)\neq 0$ for all powers of a character $\chi$ of order $m=\alpha_1$ (or in fact many such characters). To achieve this goal, we introduce  the \emph{norm of measure} (\ref{eq:normofmeasure}) which allows us to obtain a refinement of the structure theorem in \emph{op.\ cit.}, see Theorem \ref{thm:structurenew} for the precise statement. This argument suffices for the case when $G$ is cyclic (given the existence of the relevant orderly primes). For general $G$, we pick any such character $\chi_1$ of order $p_1^{\alpha_1}$ and proceed recursively by considering the product of horizontal $p$-adic $L$-functions with  $p=p_2$: 
\begin{equation}\label{eq:p2} \nu_{A\otimes \chi_1\times A\otimes \chi_1^2\times \cdots \times A,p_2}=\prod_{j=1}^{p_1^{\alpha_1}}\nu_{A\otimes \chi_1^j,p_2}\in  \Z_{p_2}(\chi_1)\llbracket G_{p_2}\rrbracket,\end{equation}
where $G_{p_2}$ is a countable product of cyclic groups of $p_2$-power order coming from orderly primes as in (\ref{eq:intromeasure}). Here a key fact for the existence of (\ref{eq:p2}) is that if there are orderly primes for $A$ modulo $p_2$ then the same is the case for the product $A\otimes \chi_1\times A\otimes \chi_1^2\times \cdots \times A$, see Lemma \ref{lem:Galoislemma}. By construction, the measure in (\ref{eq:p2}) is non-vanishing at the trivial character, and so we can apply Theorem \ref{thm:structurenew} a second time. Continuing like this for all the primary factors of $G$ yields abelian $G$-extensions $F$ such that $\rank_\Z A(F)=\rank_\Z A(\Q)=0$ as desired, using here Kato's results alluded to above. Now, to make this argument quantitative, it is most efficient  to let the last prime $p_t$ in the primary decomposition be the smallest prime divisor of $|G|$ so that when applying  Theorem \ref{thm:structurenew} at the last step of the recurrence, we conclude by a standard Tauberian argument the lower bound (\ref{eq:conclusioncorabelian}) (which loses a power of $\log X$ relative to the optimal bound (\ref{eq:wright})).

\subsubsection{On the existence of orderly primes}\label{subsubsec:orderly} The above proof strategy yields a propagation of non-vanishing principle, Theorem \ref{thm:propagG}, which relies crucially on the existence of the relevant (higher rank) horizontal $p$-adic $L$-function for \emph{all} prime divisors $p$ of the finite abelian group $G$. In view of Lemma \ref{lem:Galoislemma} alluded to above, this amounts to the existence of orderly primes which we now define. Let $\mathbb{F}_q$ be a finite field of characteristic $p$, and suppose $\rho: G_{\mathbb{Q}}\rightarrow \mathrm{GL}_n(\overline{\mathbb{F}}_p)$ is a mod $p$ Galois representation. We say that a prime $\ell$ is an \emph{orderly prime of order $m$} for $\rho$ if $\ell \equiv 1 \pmod{p^m}$, $\rho$ is unramified at $\ell$, and the image of Frobenius at $\ell$ does not have $1$ as an eigenvalue, see Definition \ref{def:orderly} below\footnote{This generalizes the notion of orderly primes from  Definitions 4.1 and 4.11 in \cite{KrizNordentoft23} to mod $p$ Galois representations of  arbitrary rank}. In Section \ref{sec:horPadic}, we show how the existence of orderly primes ensures the existence of an integral $p$-adic measure interpolating twists of $p$-power order and conductor prime to $p$ in arbitrary rank (under certain rationality assumptions related to Deligne's conjecture on critical values), which is inductive with respect to direct isobaric sum, see Section \ref{sec:higherrank}.  

Since being orderly is a Frobenius condition, the problem is amenable to group theoretic methods via the Chebotarev density theorem. For our applications we will put special emphasis on the case of products of 2-dimensional representations, but we also treat the general higher rank case (see Section \ref{sec:gln}). In particular, we obtain the following result which is responsible for the condition that ``any prime divisor of $G$ is sufficiently large'' in Theorem \ref{thm:abelianvar}.

\begin{theorem}\label{thm:non-van}
Suppose $f_1 ,\dots, f_n$ are holomorphic newforms of integral weights at least two. Let $p$ be a prime and $\mathfrak{p}_1,\dots, \mathfrak{p}_n$ prime ideals above $p$ in the Hecke fields of $f_1 ,\dots, f_n$ respectively, and let $\rho_{1},\dots, \rho_{n}$ denote the corresponding mod $p$ representations attached to $f_1 ,\dots, f_n$ respectively. When $p$ is sufficiently large, the set of orderly primes of order $m$ for $\rho_1 \oplus \cdots \oplus \rho_n$ has positive density for all $m \geq 1$.
\end{theorem}
Furthermore, when $n=2$ we show that it suffices that the mod $p$ representations are absolutely irreducible (see Corollary \ref{cor:modular-orderly} and Proposition \ref{prop:two-non-van}) which is responsible for Corollary \ref{cor:E1E2} and the last part of Theorem \ref{thm:nonvanishing}. In an appendix (Section  \ref{sec:appendix}) we show that this is \emph{not} sufficient for $n= 3$. When $n=1,2$, and $f_1,\ldots,f_n$ arise from rational elliptic curves, we give necessary and sufficient conditions for the set of orderly primes (of any order) for $\rho_1\oplus\cdots\oplus\rho_n$ has positive density, see Section \ref{sec:ellipticcurves}. 
\begin{remark}
The theory developed in \cite{KrizNordentoft23}, and refined in the present paper, is a ``horizontal'' analogue of the usual ``vertical'' $p$-adic $L$-function from Iwasawa theory interpolating twists by characters of $p$-power conductor. In the latter case, the existence of such an integral $p$-adic measure associated to a holomorphic newform $f$ of even weight is equivalent to $p$ being an \emph{ordinary} prime for $f$, meaning $p$ does not divide the $p^\mathrm{th}$ Fourier coefficient of $f$. When the weight is $\geq 4$ the existence of ordinary primes is a well-known open conjecture attributed to Serre in the literature, see the discussion in the introduction of \cite{Suh20} and the remark on page 16 of \cite{MazurTateTeitelbaum}. 
\end{remark}

\begin{remark}
   When $G$ is cyclic the arguments in \cite{FeKiKu} are recovered as the mod $p$ reduction of the horizontal $p$-adic theory sketched above, see \cite[Remark 2.1]{KrizNordentoft23}. However, their congruence method is not well-suited to deal with the case of general finite abelian $G$ because then one must ensure that the order $p_1$ twisted $L$-values are $p_2$-units for primes $p_2\neq p_1$, which seems difficult to control. 
\end{remark}

\section{Orderly primes}\label{sec:orderly} 

\subsection{Definitions and first properties}
Let $G_{\mathbb{Q}}:= \mathrm{Gal}(\overline{\mathbb{Q}}/ \mathbb{Q})$ denote the absolute Galois group, and let $\mathbb{F}_q$ denote the finite field with $q$ elements where $q=p^f$ for some prime $p$ and integer $f\geq 1$.

\begin{definition}\label{def:orderly}
Let $\rho: G_{\mathbb{Q}}\rightarrow \mathrm{GL}_n(\mathbb{F}_q)$ be a mod $p$ representation of $G_{\mathbb{Q}}$ and $m$ a positive integer. A prime $\ell$ is called \emph{orderly of order $m$ for $\rho$} if the following holds:
\begin{enumerate}
\item $\ell \equiv 1 \pmod{p^m}$;
\item $\rho$ is unramified at $\ell$;
\item $\rho(\mathrm{Frob}_{\ell})$ does not have $1$ as an eigenvalue.
\end{enumerate}
\end{definition}

\noindent
Here $\mathrm{Frob}_{\ell}$ denotes a Frobenius element over $\ell$ in $G_{\mathbb{Q}}$. When $\rho$ is unramified at $\ell$, any two Frobenius elements over $\ell$ are conjugate so condition (3) does not depend on the choice of $\mathrm{Frob}_{\ell}$. 

The purpose of this section is to investigate when the set of orderly primes of order $m$ for a given $\rho$ has positive density. By the Chebotarev density theorem, this is equivalent to the existence of an element $\tau\in\mathrm{Gal}(\overline{\mathbb{Q}}/\mathbb{Q}(\zeta_{p^m}))$ such that $\rho(\tau)$ does not have $1$ as an eigenvalue.

\begin{example}\label{ex:newform}
Let $f$ be a holomorphic newform of weight $k \geq 2$, level $N$ and central character $\eps_f$. For $\mathfrak{p}$ a prime ideal of the Hecke field $K_f$ generated by the Hecke eigenvalues of $f$, let $\rho_{f, \mathfrak{p}}: G_{\mathbb{Q}} \rightarrow \mathrm{GL}_2(\mathcal{O}_{K_{f, \mathfrak{p}}})$  be the $\mathfrak{p}$-adic Galois representation attached to $f$. Let $p$ denote the rational prime lying under $\mathfrak{p}$, and let
\begin{equation*}
\overline{\rho}_{f,\mathfrak{p}}: G_{\mathbb{Q}} \rightarrow \mathrm{GL}_2(\mathbb{F}_q)
\end{equation*}
denote the residual representation of $\rho_{f, \mathfrak{p}}$ where $q := | \mathcal{O}_{K_f}/ \mathfrak{p}|$. When $\ell$ is a prime such that $\ell \nmid p N$, $\rho_{f, \mathfrak{p}}$ is unramified at $\ell$, and if $\textrm{Frob}_{\ell} \in G_{\mathbb{Q}}$ is a Frobenius element over $\ell$, the characteristic polynomial of $\rho_{f, \mathfrak{p}}(\textrm{Frob}_{\ell})$ is
\begin{equation*}
X^2 - a_{\ell}(f) X + \eps_f(\ell)\ell^{k-1}
\end{equation*}
where $a_{\ell}(f)$ is the $\ell$\textsuperscript{th} Hecke eigenvalue of $f$. We then see that $\ell$ is an orderly prime of order $m$ for $\overline{\rho}_{f,\mathfrak{p}}$ if the following conditions hold
\begin{enumerate}
\item $\ell \equiv 1 \pmod{p^m}$;
\item $\ell$ does not divide $ p N$;
\item $a_{\ell}(f) \not\equiv 1+ \eps_f(\ell) \ell^{k-1}  \pmod{\mathfrak{p}}$.
\end{enumerate}
\noindent 
In \cite[Section 4]{KrizNordentoft23}, the above three conditions are taken as the definition of orderly primes of order $m$ for the newform $f$. However, working with a general representation $\rho: G_{\mathbb{Q}}\rightarrow \mathrm{GL}_n(\mathbb{F}_q)$ and Definition \ref{def:orderly} isolates the conditions on $\rho$ needed to ensure the existence of infinitely many orderly primes. For example, if $n=2$ and $\rho=\overline{\rho}_{f,\mathfrak{p}}$, we find conditions that - unlike those in \cite{KrizNordentoft23} - make no reference to weight $k$, the level $N$, or the character $\varepsilon_f$, but only the irreducibility of $\rho$ (compare for example Corollary \ref{cor:modular-orderly} below to \cite[Corollary 4.10]{KrizNordentoft23}).
\end{example}

\begin{remark}
For Galois representations attached to abelian varieties, orderly primes of order $1$ have already been introduced under the name \emph{silent primes} by Mazur, Rubin, and Larsen in their paper on Diophantine stability \cite[p. 12]{mazur-rubin-larsen2018}. They also consider the notion of \emph{critical primes}, for which the $1$-eigenspace of Frobenius has dimension $1$. For simple abelian varieties whose endomorphisms are all defined over the base field, the appendix by Larsen \cite[Part 3]{mazur-rubin-larsen2018} shows that for a positive proportion of primes, the associated residual Galois representations possess infinitely many silent and critical primes. Our results on orderly primes do not follow from Larsen's work since we do not have any non-CM assumptions, and we consider modular abelian varieties that are not necessarily simple.

Subgroups of $\mathrm{GL}_2(\mathbb{F}_q)$ for which every element has $1$ as an eigenvalue also play an important role in the work of Katz \cite{katz81}. However, these result are difficult to apply in our setting because we impose the condition $\ell\equiv 1\pmod{p}$ in Definition \ref{def:orderly}. 
\end{remark}

The proposition below will be the basis for proving the existence of infinitely many orderly primes. Recall that if $G$ is a group, and $g,h\in G$, we define the commutator of $g$ and $h$ as $[g,h]:=g^{-1}h^{-1}gh$. The \emph{derived subgroup of $G$} is the subgroup generated by all commutators, and we denote it by $G'$. A basic fact is that a normal subgroup of $G$ has abelian quotient if and only if it contains $G'$.  

\begin{proposition}\label{first_criteria}
Let $\rho: G_{\mathbb{Q}} \rightarrow \mathrm{GL}_n(\mathbb{F}_q)$ be a Galois representation with image $G$, and let $G'$ denote the derived subgroup of $G$. If $G'$ contains an element not having $1$ as an eigenvalue, then the set of orderly primes of order $m$ for $\rho$ has positive density for all $m \geq 1$. 
\end{proposition}

\begin{proof}
Let $\ker \rho = \mathrm{Gal}(\overline{\mathbb{Q}}/K)$ where $K$ is a finite Galois extension of $\mathbb{Q}$ so that $\rho$ defines an isomorphism $\mathrm{Gal}(K/\mathbb{Q}) \xlongrightarrow{\sim} G$. Let $\sigma \in \mathrm{Gal}(K/ \mathbb{Q})$ correspond to an element of $G'$ not having $1$ as an eigenvalue. The field $K_0 := K\cap \mathbb{Q}(\zeta_{p^m})$ is an abelian extension of $\mathbb{Q}$ contained in $K$ which means that $\mathrm{Gal}(K/K_0)$ contains the derived subgroup of $\mathrm{Gal}(K/\mathbb{Q})$, and so $\sigma \in \mathrm{Gal}(K/ K_0)$. By basic Galois theory,
\begin{equation*}
\mathrm{Gal}(K(\zeta_{p^m})/ \mathbb{Q}) \cong \left\{(\sigma_1 , \sigma_2) \in \mathrm{Gal}(K/ \mathbb{Q} ) \times \mathrm{Gal}(\mathbb{Q}(\zeta_{p^m})/ \mathbb{Q})\,:\, \sigma_1 \mid_{K_0} = \sigma_{2} \mid_{K_0} \right\}
\end{equation*}
where the isomorphism sends $ \tau \in \mathrm{Gal}(K(\zeta_{p^m})/\mathbb{Q})$ to $(\tau \mid_{K}, \tau \mid_{\mathbb{Q}(\zeta_{p^m})})$. Letting $\tau$ be the pre-image of $(\sigma, \mathrm{id}_{\mathbb{Q}(\zeta_{p^m})})$, we see that if $\ell$ is a prime unramified in $K(\zeta_{p^m})$ with Frobenius element conjugate to $\tau$, then $\ell$ is an orderly prime of order $m$. By the Chebotarev density theorem, the set of such $\ell$ has positive density.  
\end{proof}

We will also make use of the following general lemma:

\begin{lemma}\label{lem:order-one}
Let $\rho : G_{\mathbb{Q}} \rightarrow \mathrm{GL}_n(\mathbb{F}_q)$ be an $n$-dimensional mod $p$ representation. If the set of orderly primes of order $1$ for $\rho$ has positive density, then the set of orderly primes of order $m$ for $\rho$ has positive density for all $m \geq 1$.  
\end{lemma}

\begin{proof}
Let $m \geq 1$ be given. By assumption, we can find $\tau \in \mathrm{Gal}(\overline{\mathbb{Q}} / \mathbb{Q}(\zeta_{p}))$ such that $\rho(\tau)$ does not have $1$ as an eigenvalue. We have $ \tau^{p^{m-1}} \in \mathrm{Gal}(\overline{\mathbb{Q}}/ \mathbb{Q}(\zeta_{p^m}) )$ since this group is normal in $\mathrm{Gal}(\overline{\mathbb{Q}} / \mathbb{Q}(\zeta_p))$ of index $p^{m-1}$. Moreover,
\begin{equation*}
\det (\rho(\tau^{p^{m-1}})-I)=\det ((\rho(\tau)-I)^{p^{m-1}})= \det(\rho(\tau)-I)^{p^{m-1}}\neq 0
\end{equation*}
so $1$ is not an eigenvalue of $\rho(\tau^{p^{m-1}})$. Hence, $\ell$ is orderly of order $m$ for $\rho$ if $\mathrm{Frob}_{\ell}$ is conjugate to $\tau^{p^{m-1}}$, and, by the Chebotarev density theorem, a positive proportion of primes have this property.
\end{proof}

\subsection{Orderly primes in rank 2}

Because any commutator in $\mathrm{GL}_n(\mathbb{F}_q)$ has determinant $1$, Proposition \ref{first_criteria} becomes particularly useful when $n=2$, since an element of $\mathrm{SL}_2(\mathbb{F}_q)$ has $1$ as an eigenvalue if and only if it is unipotent. Hence there are infinitely many orderly primes for a representation $\rho: G_{\mathbb{Q}}\rightarrow \mathrm{GL}_2(\mathbb{F}_q)$ if the derived subgroup of the image does not consist of unipotent matrices. Moreover, since we work in characteristic $p$, a subgroup $U\leq \mathrm{GL}_n(\mathbb{F}_q)$ consists of unipotent matrices if and only if it is a $p$-group.

Recall that a representation $\rho: G_{\mathbb{Q}} \rightarrow \mathrm{GL}_n(\mathbb{F}_q)$ is \emph{absolutely irreducible} if it is irreducible and remains irreducible after base change to $\overline{\mathbb{F}}_q$. When $n=2$, and $\rho$ is absolutely irreducible, we prove that the set of orderly primes of order $m$ has positive density for all $m\geq 1$. As explained above, it is enough to prove the following lemma.

\begin{lemma}\label{unipotent}
Let $n \geq 2$, and suppose $\rho: G_{\mathbb{Q}} \rightarrow \mathrm{GL}_n(\mathbb{F}_q)$ is absolutely irreducible and has image $G$. Then the derived subgroup $G'$ does not consist of unipotent matrices.  
\end{lemma}

\begin{proof}
Suppose for the sake of contradiction that $G'$ consists of unipotent matrices. Since $n\geq 2$, and $\rho$ is absolutely irreducible, $G$ cannot be abelian so the derived subgroup $G'$ is non-trivial. Let $V:=(\mathbb{F}_q^n)^{G'}$ be the space of vectors fixed by all elements of $G'$. We have $V\neq \mathbb{F}_q^n$ because $G'\neq\{1\}$. Since $G'$ consists of unipotent matrices, $G'$ is a finite $p$-group so by \cite[Section 8.3, Proposition 26]{serre1977linear} it follows that $V\neq 0$. We claim that $V$ is $G$-stable which would lead to the desired contradiction since $\rho$ is irreducible. Let $v\in V$ and $g\in G$. If $h\in G'$, we have $[h,g]v=v$ where $\left[ h,g \right]= h^{-1} g^{-1}  h g$  is the commutator of $h$ and $g$. Since $hv=v$, this rewrites as $hg v = g v$ so $gv \in V$ since $h\in G'$ was arbitrary. Therefore, $V$ is $G$-stable.
\end{proof}

Hence we deduce the following result.

\begin{corollary}
Suppose $\rho : G_{\mathbb{Q}} \rightarrow \mathrm{GL}_2(\mathbb{F}_q)$ is absolutely irreducible. Then the set of orderly primes of order $m$ for $\rho$ has positive density for all $m \geq 1$. 
\end{corollary}

\begin{proof}
Let $G$ be the image of $\rho$. By Proposition \ref{first_criteria}, it is enough to show that $G'$ contains an element not having $1$ as an eigenvalue. Since $G' \subset \mathrm{SL}_2(\mathbb{F}_q)$, this is equivalent to $G'$ not consisting of unipotent matrices which we know is the case by Lemma \ref{unipotent}. 
\end{proof}

Recall that for $p\neq 2$, we say that a Galois representation $\rho: G_{\mathbb{Q} } \rightarrow \mathrm{GL}_n(\mathbb{F}_q)$ is \emph{odd} if $\det \rho(c) = -1$ where $c \in G_{\mathbb{Q} }$  denotes a complex conjugation. It is well-known that an odd and irreducible representation $\rho: G_{\mathbb{Q}}\rightarrow \mathrm{GL}_2(\mathbb{F}_q)$ is absolutely irreducible, see for example \cite[p. 198]{Serre1987}. Since the $p$-adic Galois representation attached to a holomorphic newform is odd, we deduce:

\begin{corollary}\label{cor:modular-orderly}
Suppose $\rho = \overline{\rho}_{f, \mathfrak{p}}$ for some holomorphic newform $f$ and odd prime ideal $\mathfrak{p}$ of $\mathcal{O}_{K_f}$. If $\rho$ is irreducible, then the set of orderly primes of order $m$ for $\rho$ has positive density for all $m \geq 1$.
\end{corollary}

\subsection{Simultaneous orderly primes}

In applications to simultaneous non-vanishing of $L$-values, we must prove the existence of infinitely many primes that are orderly for multiple representations $\rho_1$, \dots, $\rho_n: G_{\mathbb{Q}}\rightarrow \mathrm{GL}_2(\mathbb{F}_q)$ at the same time. Equivalently, if $\rho= \oplus_{i=1}^n \rho_i$, we must prove the existence of infinitely orderly primes for $\rho$. We show that absolute irreducibility of $\rho_1$ and $\rho_2$ is sufficient for $n=2$, but additional assumptions are needed for $n\geq 3$.\\

We start with case $n=2$ and prove the following lemma:

\begin{lemma}\label{lem:prod}
Let $G_1 , G_2 \leq \mathrm{GL}_2(\mathbb{F}_q)$ be non-unipotent subgroups, and suppose that $G \leq G_1 \times G_2$ is a subgroup that surjects onto each factor. Then $G$ contains an element $(g_1 , g_2)$ such that $g_1$ and $g_2$ are not unipotent.
\end{lemma}

\begin{proof}
Recall that an element of $\GL_2(\F_q)$ is unipotent if and only if its order is a power of $p$. Choose $g=(g_1 , g_2) \in G$ such that $g_1$ is not unipotent. If $g_2$ is not unipotent, then we are done. Otherwise, we can replace $g$ by $g^{p^r}$ for some $r>0$ and assume that $g_2=1$. Now choose $g ' =(g_1 ', g_2 ') \in G $ such that $g_2 '$ is not unipotent. If $g_1 '$ is not unipotent, we are done. Otherwise, we can, as before, assume that $g_1 '=1$. Then both components of $g g ' =(g_1 , g_2 ')$ are non-unipotent.  
\end{proof}

\noindent
We can prove that the set orderly primes for $\rho_1\oplus\rho_2$ has positive density provided that $\rho_1$ and $\rho_2$ are absolutely irreducible.

\begin{proposition}\label{prop:two-non-van}
Suppose $\rho_1 , \rho_2 : G_{\mathbb{Q}} \rightarrow \mathrm{GL}_2(\mathbb{F}_q)$ are absolutely irreducible, and let $\rho = \rho_1 \oplus \rho_2$. Then the set of orderly prime of order $m$ for $\rho$ has positive density for all $m\geq 1$. 
\end{proposition}

\begin{proof}
Let $G_i$ denote the image of $\rho_i$ for $i=1,2$, and $G$ be the image of $\rho$. The derived subgroup of $G$ is a subgroup of $G_1'\times G_2'$ that surjects onto each factor. Hence, by Proposition \ref{first_criteria} and Lemma \ref{lem:prod} above, it is enough to show that $G_1$ and $G_2$ are not unipotent. But $\rho_1$ and $\rho_2$ are absolutely irreducible, so this follows from Lemma \ref{unipotent}.  
\end{proof}

We now move on to the case when $n\geq 3$ where the main goal is to prove Theorem \ref{thm:non-van} from the introduction. We start with the following result which shows that Proposition \ref{prop:two-non-van} does not generalize to three or more representations.

\begin{proposition}\label{counter_example}
Suppose $q\equiv 1\pmod{4}$. Then there are odd and absolutely irreducible representations $\rho_1,\rho_2,\rho_3:G_{\mathbb{Q}}\rightarrow \mathrm{GL}_2(\mathbb{F}_q)$ such that for all $m\geq 1$, there are no orderly primes of order $m$ for $\rho_1\oplus\rho_2\oplus\rho_3$.
\end{proposition}

\noindent
The proof is by an explicit construction that we postpone to the appendix (Section \ref{sec:appendix}). The point of choosing $\rho_1$, $\rho_2$ and $\rho_3$ to be odd is that by Serre's modularity conjecture (now a theorem due to Khare and Wintenberger \cite{KhareWintenberger2009a, KhareWintenberger2009b}), they are the residual representations attached to some newforms of weight at least $2$, and this is exactly the case, we are interested in.

However, in the special case where $q=2$, we can actually generalize Proposition \ref{prop:two-non-van} to $n=3$. 

\begin{proposition}\label{n=3p=2}
Let $\rho_1,\rho_2,\rho_3:G_{\mathbb{Q}}\rightarrow \mathrm{GL}_2(\mathbb{F}_2)$ be irreducible representations. Then for all $m\geq 1$, the set of orderly primes of order $m$ for $\rho_1\oplus\rho_2\oplus\rho_3$ has positive density among all primes.
\end{proposition}
\begin{proof}
By irreducibility, the image of each $\rho_i$ contains the unique subgroup 
$H:=\langle \begin{psmallmatrix}
    0& 1\\ 1&1
\end{psmallmatrix}\rangle$ of order $3$ in $\GL_2(\F_2)$. Consider the intersection of  the image $G$ of $\rho_1\oplus\rho_2\oplus\rho_3$ with $H^3$, which we identify with a subspace $V$ of $\F_3^3$. We claim that it is enough to prove that $V$ contains an element $g=(g_1,g_2,g_3)$ with each $g_i$ non-zero. Indeed, if this holds, we can choose $\tau\in G_{\mathbb{Q}}$ such that $(\rho_1\oplus\rho_2\oplus\rho_3)(\tau)=g$. No non-trivial element of $H$ has $1$ as an eigenvalue, so, by the Chebotarev density theorem, the set orderly primes of order $1$ for $\rho_1\oplus\rho_2\oplus \rho_3$ indeed has positive density. By Lemma \ref{lem:order-one}, the same is true for the set of orderly primes of order $m$ for all $m\geq 1$.

To construct the desired element $(g_1,g_2,g_3)$, we first observe that the projection of $V$ onto each factor is surjective. Indeed, suppose $g=(g_1,g_2,g_3)$ lies in the image of $\rho_1\oplus\rho_2\oplus\rho_3$ with $g_1$ a non-trival element of $H$. If $g$ does not lie in $V$, then $g_2$ or $g_3$ have order $2$ so $g^2$ must lie in $V$, and $g_1^2$ is still a non-trivial element of $H$. Hence $V$ surjects onto the first factor, and by an identical argument, it also surjects onto the other factors.

For $i=1,2,3$, let $\pi_i:V\rightarrow\mathbb{F}_3$ denote the projection onto the $i$\textsuperscript{th} component. Since each $\pi_i$ is surjective, its kernel must have size $|V|/3$. Moreover, the kernels of $\pi_1$, $\pi_2$ and $\pi_3$ are not pairwise disjoint since they all contain the zero-vector. Hence $|\cup_i\ker\pi_i|<3\cdot |V|/3=|V|$, so indeed there is an element of $V$ whose coordinates are all non-zero.             
\end{proof}

By Proposition \ref{counter_example}, it follows that Theorem \ref{thm:non-van} cannot be proved only using irreducibility of the mod $p$ representations $\rho_1,\dots,\rho_n$ for $p$ sufficiently large. Instead, we use modularity of these representations to prove that the proportion of unipotent elements in the derived subgroups of their images tends to zero as $p \rightarrow \infty$. Using Proposition \ref{first_criteria}, Theorem \ref{thm:non-van} then follows by a counting argument. The first step in these arguments is the following proposition. 

\begin{proposition}\label{far-from-unipotent}
Let $f$ be a holomorphic newform of weight at least $2$. Suppose $p$ is a prime, and $\mathfrak{p}$ is a prime above $p$ in the Hecke field of $f$. Let $\overline{\rho}_{f, \mathfrak{p}}$ denote the corresponding mod $p$ representation attached to $f$, and write $G$ for the image of $\overline{\rho}_{f, \mathfrak{p}}$ and $G'$ for the derived subgroup of $G$. Then there is a constant $c_f>0$ (depending only on $f$) such that the proportion of unipotent elements in $G'$ is at most $c_f /(p-1)$.
\end{proposition}

\noindent
There will be two distinct cases depending on whether $f$ has complex multiplication. In Lemma \ref{CM-case}, we handle the CM case, and in Lemma \ref{non-CM-case} we handle the non-CM case. Before proving these lemmas, we explain how the above proposition implies Theorem \ref{thm:non-van}.  

\begin{proof}[Proof of Theorem \ref{thm:non-van} assuming Proposition \ref{far-from-unipotent}]
Let $G_i$ denote the image of $\rho_i$ for $i=1,\dots,n$, and let 
\begin{equation*}
G: = \mathrm{Im}(\rho_1 \oplus \cdots \oplus \rho_n) \leq G_1 \times \cdots \times G_n.    
\end{equation*}
The derived subgroup $G'$ is a subgroup of $G_1 ' \times \cdots \times G_n'$ that surjects onto each factor. If $\pi_i: G' \twoheadrightarrow G_i ' $ denotes projection onto the $i$\textsuperscript{th} factor, it is enough to show that $G'$ contains an element $g$ such that $\pi_i(g)$ is not unipotent for all $i=1,\dots,n$ and $p$ sufficiently large. Let $U_i$ denote the set of unipotent elements in $G_i'$.  By Proposition \ref{far-from-unipotent}, we can find a constant $c>0$ (depending only on $f_1 ,\dots, f_n$) such that $\lvert U_i \rvert/ \lvert G_i' \rvert \leq c/(p-1)$ for $i=1,\dots,n$. Since each $\pi_i$ is surjective, it follows that
\begin{equation*}
\frac{\lvert \pi_i^{-1} (U_i) \rvert}{ \lvert G' \rvert} = \frac{\lvert U_i \rvert}{ \lvert G_i ' \rvert} \leq \frac{c}{p-1}
\end{equation*}
for all $i=1,\dots,n$, and hence
\begin{equation*}
\frac{\lvert \cup_i \pi_i^{-1}(U_i) \rvert}{\lvert G' \rvert} \leq \sum_{i=1}^n\frac{|\pi_i^{-1}(U_i)|}{|G'|} \leq\frac{cn}{p-1}<1
\end{equation*}
for all $p$ sufficiently large as desired.
\end{proof}

\begin{lemma}\label{CM-case}
  Let $f$ be a newform of weight $k\geq 2$  with complex multiplication. Let $p$ be a prime, $\mathfrak{p}$ a prime above $p$ in the Hecke field of $f$ and $\overline{\rho}_{f,\mathfrak{p}}$ the corresponding mod $p$ representation. Let $G$ denote the image of $\overline{\rho}_{f,\mathfrak{p}}$. Then for $p$ sufficiently large, $\overline{\rho}_{f,\mathfrak{p}}$ is absolutely irreducible, and the proportion of unipotent elements in the derived subgroup $G'$ is at most $(k-1)/(p-1)$.   
\end{lemma}
\begin{proof}
By \cite[Proposition 4.4]{Ribet1977} there exists an imaginary quadratic field $F/\Q$ so that the restriction of the associated $\mathfrak{p}$-adic representation $\rho_{f,\mathfrak{p}}$ to $G_F$ is abelian and semisimple. Since $F$ is imaginary quadratic, it follows that complex conjugation $c\in G_\Q$ is a representative for the non-trivial coset of $G_\Q/G_F$, and the image $\rho_{f,\mathfrak{p}}(c)\in GL_2(\mathcal{O}_{K_{f, \mathfrak{p}}})$ is an order two matrix with determinant $-1$. These considerations imply (see the proof of \cite[Theorem 4.5]{Ribet1977}) that there exists an algebraic Hecke character $\psi$ of $F$, so that we have the following isomorphism over $\overline{\mathbb{Q}}_p$:
\begin{equation}
\label{eq:psic} (\rho_{f,\mathfrak{p}})_{|G_F}\cong  \psi \oplus \psi^c  ,
\end{equation}
where $\psi^c(g):=\psi(cgc)$ for $g\in G_F$. Here we are freely identifying $\psi$ with a one-dimensional $\mathfrak{p}$-adic representation of $G_F$ via class field theory. It is however also convenient to think of $\psi$ (classically) as a homomorphism $\mathcal{I}_\mathfrak{m}\rightarrow E^\times$ with domain the group of (proper) fractional ideals of $F$ coprime to an ideal $\mathfrak{m}\subset \mathcal{O}_F$ (the conductor) and taking values in a number field $E/\Q$, see \cite{Weil56}. Note that under this identification it holds that $\psi^c(\mathfrak{a})=\psi(\overline{\mathfrak{a}})$ for $\mathfrak{a}\in \mathcal{I}_\mathfrak{m}$.

Since the quotient $G_\Q/G_F\cong \Z/2$ is abelian, the derived subgroup of $G_\Q$ is contained in $G_F$. In particular, the derived subgroup $G'\leq \GL_2(\F_q)$ of the image of $\overline{\rho}_{f,\mathfrak{p}}$ is contained in the subgroup of diagonal matrices with determinant $1$ (in the basis corresponding to (\ref{eq:psic})). This means that the identity is the only unipotent element in $G'$ and so it suffices for the second part of the lemma to show that $G'$ has size at least $(p-1)/(k-1)$. By (\ref{eq:psic}) the image of $\overline{\rho}_{f,\mathfrak{p}}$, in the basis corresponding to (\ref{eq:psic}), contains 
$$  
\left\{ \begin{pmatrix} \psi(\mathfrak{l})& 0\\ 0& \psi(\overline{\mathfrak{l}}) \end{pmatrix}\modulo \mathfrak{p}: (\ell)=\mathfrak{l}\overline{\mathfrak{l}}\text{ split in $F$} \right\},$$
 since $c\,\Frob_\mathfrak{l}\, c=\Frob_{\overline{\mathfrak{l}}}$. We conclude that for $(\ell)=\mathfrak{l}\overline{\mathfrak{l}}$ a split prime in $F$ and coprime to the conductor $\mathfrak{m}$ the derived subgroup $G'$ contains
\begin{equation}\overline{\rho}_{f,\mathfrak{p}}\left( \Frob_\mathfrak{l} \,c\, (\Frob_\mathfrak{l})^{-1}\, c  \right) = \begin{pmatrix} \psi(\mathfrak{l})& 0\\ 0& \psi(\overline{\mathfrak{l}}) \end{pmatrix} \begin{pmatrix} \psi(\overline{\mathfrak{l}})^{-1} & 0\\ 0& \psi(\mathfrak{l})^{-1} \end{pmatrix} = \begin{pmatrix} \frac{\psi}{\psi^c}(\mathfrak{l})& 0\\ 0& \frac{\psi^c}{\psi}(\mathfrak{l}) \end{pmatrix}\mod \mathfrak{p},
\end{equation}
(note here that $\Frob_\mathfrak{l},\Frob_{\overline{\mathfrak{l}}} \in G_K\leq G_\Q$ are both Frobenius elements for $\ell$). Recall also from basic representation theory \cite[Section  7.2.1]{Serre77} that $\overline{\rho}_{f,\mathfrak{p}}$ is isomorphic to the induction from $G_K$ to $G_\Q$ of $\psi \modulo \mathfrak{p}$  which is irreducible if and only if $\psi \not\equiv \psi^c\modulo \mathfrak{p}$ (see e.g.\ the  discussion in \cite[Section  3.1]{KedlayaMedved}). In conclusion, it suffices for both parts to lower bound the size of the image of the mod $p$ representation associated to $\psi/\psi^c$ (since it is trivial on (Frobenius elements of) inert primes).  

Since $f$ is of weight $k$ we have that $\psi((\alpha))=\alpha^{k-1}$ for $\alpha\in \mathcal{O}_F $ satisfying  $\alpha\equiv 1 \modulo \mathfrak{m}$ (see \cite[p.\ 35]{Ribet1977}). Note that for $(p)=\mathfrak{p}_1\mathfrak{p}_2$ split in $F$ and sufficiently large, the map
\begin{align}
    \left\{  \alpha\in \mathcal{O}_F : \alpha \equiv 1 \modulo \mathfrak{m}, ((\alpha), (pN_f))=1 \right\}\rightarrow \F_p^\times\times \F_p^\times ,\quad  \alpha\mapsto(\alpha\modulo \mathfrak{p}_1,\overline{\alpha}\modulo \mathfrak{p}_1) 
\end{align}
is surjective by the Chinese Reminder Theorem (using here that complex conjugation defines an isomorphism $\Oo_F/\mathfrak{p}_1\cong \Oo_F/\mathfrak{p}_2$). Thus the mod $\mathfrak{p}$ image of $\psi/\psi^c$ contains $(\F_p^\times)^{k-1}$ which gives the wanted conclusion in this case. For $p$ inert in $F$ and sufficiently large we conclude again by Chinese Reminder that 
the image of 
\begin{align}
    \left\{  \alpha\in \mathcal{O}_F : \alpha \equiv 1 \modulo \mathfrak{m}, ((\alpha), (pN_f))=1 \right\}\rightarrow \F_{p^2}^\times ,\quad  \alpha\mapsto (\alpha \modulo p), 
\end{align}
is surjective. Since complex conjugation induces the non-trivial automorphism of $\F_{p^2}$ we conclude that   $\alpha/\overline{\alpha}\equiv \alpha^{1-p}\modulo p$, this implies that the mod $\mathfrak{p}$ image of $\psi/\psi^c$ contains $(\F_{p^2}^\times)^{(k-1)(p-1)} $, which is of size at least $(p+1)/(k-1)$. This yields the wanted conclusion in this case as well.
\end{proof}

\begin{lemma}\label{non-CM-case}
Let $f$ be a newform without complex multiplication of weight at least $2$. Let $p$ be a prime, $\mathfrak{p}$ a prime above $p$ in the Hecke field of $f$ and $\overline{\rho}_{f,\mathfrak{p}}$ the corresponding mod $p$ representation. Let $G$ denote the image of $\overline{\rho}_{f,\mathfrak{p}}$. Then for $p$ sufficiently large, $\overline{\rho}_{f,\mathfrak{p}}$ is absolutely irreducible and the proportion of unipotent elements in the derived subgroup $G'$ is at most $1/(p-1)$. 
\end{lemma}

\begin{proof}
Let $\mathbb{F}_{q}$ be the residue field of $\mathfrak{p}$ so that $G$ is a subgroup of $\mathrm{GL}_2(\mathbb{F}_q)$. By Ribet's generalization of Serre's open image theorem \cite[Theorem 3.1]{Ribet1985}, it follows that $G$ contains $\mathrm{SL}_2(\mathbb{F}_p)$ for $p$ sufficiently large and so, in particular, $\overline{\rho}_{f,\mathfrak{p}}$ is absolutely irreducible. Furthermore, since $\mathrm{SL}_2(\mathbb{F}_p)$ is perfect for $p\geq 5$ \cite{special}, $G'$ contains $\mathrm{SL}_2(\mathbb{F}_p)$ when $p$ is sufficiently large. In particular $G'$ contains $N:= \{(\begin{smallmatrix}
 1&x\\0&1   
\end{smallmatrix})\,:\,x \in \mathbb{F}_p\}$. By Sylow's theorems, $N$ is contained in a Sylow $p$-subgroup $P$ of $G'$.  Any unipotent element of $G'$ has order a power of $p$ and so is contained in a Sylow $p$-subgroup of $G$. By Sylow's theorems, any Sylow $p$-subgroup of $G'$ is conjugate to $P$ so the total number of Sylow $p$ -subgroups is equal to the index $[G':N_{G'}(P)]$ where $N_{G'}(P)$ is the normalizer of $P$ in $G'$. Hence the proportion of unipotent elements in $G'$ is bounded by 
\begin{equation*}
\frac{[G':N_{G'}(P)]|P|}{|G'|}=\frac{|P|}{|N_{G'}(P)|}.
\end{equation*}
The task is now to show that $|N_{G'}(P)|\geq (p-1)|P|$.

Since $P$ is a $p$-group, there is a non-zero vector in $\mathbb{F}_q^2$ fixed by all elements of $P$ \cite[Section 8.3, Proposition 26]{serre1977linear}. The only vectors fixed by $N$ are scalar multiples of $(1,0)$. It follows that $(1,0)$ is fixed by all elements of $P$ so $P$ must be of the form $\{(\begin{smallmatrix}
1&v\\0&1   
\end{smallmatrix})\,:\, v \in V\}$ where $V$ is an $\mathbb{F}_p$-subspace of $\mathbb{F}_q$ (using that $P\subset \SL_2(\F_q)$). Since $\mathrm{SL}_2(\mathbb{F}_p)\subset G'$, we find that $G'$ contains $A:=\{(\begin{smallmatrix}a&0\\0&a^{-1}\end{smallmatrix})\,:\,a \in \mathbb{F}_p^{\times}\}$. Clearly, $A$ normalizes $P$ and intersects $P$ trivially so $|N_{G'}(P)|\geq |A||P|=(p-1)|P|$ as desired.
\end{proof}

\subsection{The case of elliptic curves}\label{sec:ellipticcurves}

We now consider Galois representations attached to rational elliptic curves, where several refinements of the preceding results are possible. If $E$ is a rational elliptic curve, and $p$ is prime, we start by giving necessary and sufficient conditions for the residual representation $\overline{\rho}_{E,p}: G_{\mathbb{Q}} \rightarrow \mathrm{GL}_2(\mathbb{F}_p)$ to have a positive proportion of orderly primes (of any order). The new input is Grothendieck's $\ell$-adic monodromy theorem \cite{sga7-1} which gives a description of the action of inertia at primes $\ell \neq p$ on the $p$-adic Tate-modulo of $E$ in terms of the reduction of $E$ at $\ell$. Our main result is the following theorem:

\begin{theorem}\label{thm:elliptic-orderly}
Let $E$ be a rational elliptic curve and $p$ a prime. Then the following holds:
\begin{enumerate}[(i)]
\item If $p =2$, then the set of orderly primes of order $m$ for $\overline{\rho}_{E,p}$ has positive density for all $m \geq 1$ if and only if $E$ has no rational point of order $2$.
\item If $p=3$, then the set of orderly primes of order $m$ for $\overline{\rho}_{E,p}$ has positive density for all $m \geq 1$ if and only if $E $ is not isogenous (over $\mathbb{Q}$) to a rational elliptic curve with a rational point of order $3$.
\item If $p \geq 5$, then the set of orderly primes of order $m \geq 1$ for $\overline{\rho}_{E,p}$ has positive density for all $m \geq 1$ if and only if $\overline{\rho}_{E,p}$ is irreducible, or $E$ has a prime $\ell \neq p$ of additive reduction. 
\end{enumerate}
\end{theorem}

\noindent
We recall that $E$ has additive (or unstable) reduction at a prime $\ell$ if the group of non-singular $\overline{\mathbb{F}}_{\ell}$-points on the reduced curve modulo $\ell$ is isomorphic to the additive group of $\overline{\mathbb{F}}_{\ell}$. Additive reduction at $\ell$ is equivalent to $\ell^2$ dividing the conductor of $E$ \cite[Ch. IV, \S 10]{silverman94}.\\

\noindent
The above theorem allows us to give many concrete examples of rational elliptic curves $E$ such that all associated mod $p$ representations have a positive proportion of orderly primes (of any order). Indeed, if a rational elliptic curve $E$ is listed in the LMFDB-database, one can easily read off from its data, whether conditions (i)-(iii) in the above theorem hold.

For example, the LMFDB-database contains more than two million elliptic curves whose isogeny class have size $1$, and these only have irreducible residual representations. More than 740.000 of them have (analytic) rank $0$, so they meet the conditions of Corollary \ref{cor:E1E2}.

We can also give examples of elliptic curves with reducible residual representations that still have a positive proportion of orderly primes for each $p$. For example, if $E$ is the elliptic curve $y^2+y=x^3-66x-212$ with LMFDB-label $1089.g2$, then $E$ has rank $0$ and conductor $3^2 \cdot 11^2$. The mod $p$ representation $\overline{\rho}_{E,p}$ is irreducible for all $p\neq 11$. Since $E$ has additive reduction at $3$, it follows that for all primes $p$, the set of orderly primes (of any order) for $\overline{\rho}_{E,p}$ has positive density. We remark that $E$ has complex multiplication, which is a condition that is usually excluded in Diophantine stability results.\\

We also consider the case of two or more elliptic curves, where we can explicate what "sufficiently large" means in Theorem \ref{thm:non-van}. Here the new input is a classification of the possible images of $\overline{\rho}_{E,p}$ due to Zywina \cite{Zywina2015} and a theorem of Mazur \cite[Theorem 4]{mazur78}. To state our result, we write $j_E$ for the $j$-invariant attached to an elliptic curve $E$, and define the following of ``bad'' values for $(p, j_E)$ where the mod $p$ representation attached to $E$ is reducible \cite{Zywina2015}:
\begin{equation*}
\begin{split}  
  S := \{ & (17,-17 \cdot 373^3/2^{17}),(17,-17^2 \cdot 101^3 /2),(37,-7 \cdot 11^3),(37,-7 \cdot 137^3 \cdot 2083^3),\\
       &(19,-2^{15}\cdot 3^3),(43, -2^{18} \cdot 3^3 \cdot 5^3),(67, - 2^{15} \cdot 3^3 \cdot 5^3 \cdot 11^3),(163, -2^{18}\cdot 3^3 \cdot 5^3 \cdot 23^{3} \cdot 29^3) \}.
\end{split}
\end{equation*}
We then have the following theorem:

\begin{theorem}\label{thm:ordrlyellipticcurves}
Let $E_1,\dots, E_n$ be rational elliptic curves, and $p$ a prime. Then the following holds:
\begin{enumerate}[(i)]
\item If $n =2$, then the set of orderly primes of order $m$ for $\overline{\rho}_{E_1,p} \oplus \overline{\rho}_{E_2,p}$ has positive density for all $m \geq 1$ if and only if the same holds for the individual representations $\overline{\rho}_{E_1 , p}$ and $\overline{\rho}_{E_2 ,p}$.
\item If $(p, j_{E_i} ) \notin S$ for $i=1,\ldots, n$, and $p \geq \max \left\{ 17,3n+2 \right\} $, then the set of orderly primes of order $m$ for $\overline{\rho}_{E_1 , p} \oplus \cdots \oplus \overline{\rho}_{E_n ,p}$ has positive density for all $m \geq 1$. In particular, if $p \geq \max \left\{ 167, 3n +2 \right\}$, this conclusion holds without any conditions on the $j$-invariants of $E_1 ,\ldots, E_n$.  
\item If $E_1 ,\ldots, E_n$ are semi-stable, and $p \geq \max \left\{ 11, n+1 \right\}$, then the set of orderly primes of order $m$ for $\overline{\rho}_{E_1 , p} \oplus \cdots \oplus \overline{\rho}_{E_n ,p}$ has positive density for all $m \geq 1$. 
\end{enumerate}
\end{theorem}

\noindent
Recall that $E$ is semi-stable if it has no primes of additive reduction, or equivalently if $E$ has square-free conductor. 

\subsubsection{The proof of Theorem \ref{thm:elliptic-orderly}} Let $E$ be a rational elliptic curve, and $p$ a prime. We have already seen that if the mod $p$ representation $\overline{\rho}_{E,p}$ is irreducible, then the set of orderly primes (of any order) for $\rho$ has positive density. Hence it remains to study the case when $\overline{\rho}_{E,p}$ is reducible. By Lemma \ref{lem:order-one}, we only need to determine when $\overline{\rho}_{E,p}$ has a positive proportion of orderly primes of order $1$.

Recall that the determinant of $\overline{\rho}_{E,p}$ is the mod $p$ cyclotomic character $\chi_{\mathrm{cyc},p}: G_{\mathbb{Q}} \rightarrow \mathbb{F}_p^\times$. We start by observing some simple consequences of this fact that reduce the problem to understanding the action of inertia on $E \left[ p \right]$ at primes different from $p$. 

\begin{lemma}\label{lem:orderly-condition}
 Let $\rho : G_{\mathbb{Q}} \rightarrow \mathrm{GL}_2(\mathbb{F}_p)$ be a mod $p$ representation whose determinant is the mod $p$ cyclotomic character. Then the set of orderly primes of order $1$ for $\rho$ has positive density if and only the intersection of the image of $\rho$ with $\mathrm{SL}_2(\mathbb{F}_p)$ is not unipotent.
\end{lemma}

\begin{proof}
By the Chebotarev density theorem, the set of orderly primes of order $1$ for $\rho$ has positive density if and only if there is an element $\tau \in \mathrm{Gal}(\overline{\mathbb{Q}}/\mathbb{Q}(\zeta_p))$ such that $\rho(\tau)$ does not have $1$ as an eigenvalue. Since $\det \rho = \chi_{\mathrm{cyc},p}$, we have $\tau \in \mathrm{Gal}(\overline{\mathbb{Q}} / \mathbb{Q}(\zeta_p))$ if and only if $\rho(\tau) \in \mathrm{SL}_2(\mathbb{F}_p)$. Now the proof is complete since an element of $\mathrm{SL}_2(\mathbb{F}_p)$ has $1$ as an eigenvalue if and only if it is unipotent.
\end{proof}

\begin{lemma}\label{lem:orderly-unipotent}
Let $\rho: G_{\mathbb{Q}} \rightarrow \mathrm{GL}_2(\mathbb{F}_p)$ be a reducible mod $p$ representation whose determinant is the mod $p$ cyclotomic character. Then the set of orderly primes of order $1$ for $\rho$ has positive density if and only if there is a prime $\ell \neq p$ such that $\rho(I_{\ell})$ is not unipotent for some (hence any) inertia subgroup $I_{\ell} \leq G_{\mathbb{Q}}$ above $\ell$. 
\end{lemma}

\begin{proof}
We may assume that $\rho= \left(
\begin{smallmatrix}
\chi_1&\ast\\
0&\chi_2 \\
\end{smallmatrix}
\right)$ for some characters $\chi_1 , \chi_2 : G_{\mathbb{Q}} \rightarrow \mathbb{F}_p^\times$ with $\chi_1 \chi_2 = \chi_{\mathrm{cyc},p}$. Then $\rho(I_{\ell})$ is unipotent for all $\ell \neq p$ if and only if $\chi_1$ and $\chi_2$ are unramifed at all $\ell \neq p$. This last statement is equivalent to $\chi_1$ and $\chi_2$ factoring through $\mathrm{Gal}(\mathbb{Q}(\zeta_{p^m})/\mathbb{Q})$ for some $m \geq 1$ which in turn is equivalent to $\rho$ not having any orderly primes of order $m$ for some $m \geq 1$. By Lemma \ref{lem:order-one}, this is equivalent to $\rho$ not having any orderly primes of order $1$. 
\end{proof}

\noindent
The above lemma provides the link between orderly primes and the geometry of the elliptic curve $E$. Let
\begin{equation*}
T_p(E):= \lim_{\substack{
\longleftarrow \\ n 
}} E \left[ p^n \right]
\end{equation*}
denote the $p$-adic Tate-module of $E$. By Grothendieck's $\ell$-adic monodromy theorem, for $\ell\neq p$, a finite index subgroup of $I_\ell$ acts unipotently on $T_p(E)$, and $I_{\ell}$ itself acts unipotently if and only if $E$ is semi-stable at $\ell$, see \cite[Expos\'{e} IX, Th\'{e}oreme 3.6 and Corollaire 3.8]{sga7-1}. We can now prove the following proposition which is the main input to Theorem \ref{thm:elliptic-orderly}:

\begin{proposition}\label{prop:orderly-additive}
Let $E$ be a rational elliptic curve, and suppose $p \geq 5$ is a prime such that $\overline{\rho}_{E,p}$ is reducible. Then the set of orderly primes for $\overline{\rho}_{E,p}$ has positive density if and only if $E$ has a prime $\ell\neq p$ of additive reduction.
\end{proposition}

\begin{proof}
If $\ell  \neq p$ is prime, the inertia group $I_{\ell}$ acts unipotently on $T_p(E)$ if and only if $E$ is semi-stable at $p$ \cite[Expos\'{e} IX, Corollaire 3.8]{sga7-1}. By Lemma \ref{lem:orderly-unipotent}, it is therefore enough to show that if $E$ has additve reduction at $\ell$, then $I_{\ell}$ does not act unipotently on $E \left[ p \right]$. By the semi-stable reduction theorem \cite[Expos\'{e} IX, Th\'{e}oreme 3.6]{sga7-1}, $E$ becomes semi-stable over a finite extension of $\mathbb{Q}_{\ell}$, so there is a subgroup $I_0 \leq I_{\ell}$ of finite index that acts unipotently on $T_p(E)$. Since $E$ is not semi-stable at $\ell$, there is an element $\sigma \in I_{\ell} \setminus I_0$ that does not act unipotently on $T_p(E)$, and we now show that $\overline{\rho}_{E,p}(\sigma)\in\mathrm{GL}_2(\mathbb{F}_p)$ is not unipotent.

The Jordan normal form of $\rho_{E,p}(\sigma)$ is either diagonal or $\left(
\begin{smallmatrix}
\lambda&1\\
0&\lambda \\
\end{smallmatrix}
\right)$. Because $\left[ I_{\ell} : I_0 \right]$ is finite, $\rho_{E,\rho}(\sigma)^M$ is unipotent for some $M>0$, so the eigenvalues of $\rho_{E,p}(\sigma)$ must be roots of unity. Moreover, they lie in a quadratic extension of $\mathbb{Q}_p$, so, since $\zeta_p$ has degree $p-1$ over $\mathbb{Q}_p$, and $p>3$, they must have order coprime to $p$. Hence, if $\rho_{E,p}(\sigma)$ is diagonalizable, it has finite order coprime to $p$. When $p>2$, the kernel of the reduction map $\mathrm{GL}_2(\mathbb{Z}_p) \twoheadrightarrow \mathrm{GL}_2(\mathbb{F}_p)$ contains no non-trivial elements of finite order \cite[Ch. IV, Exercise 4.38]{silverman94} so it follows that $\overline{\rho}_{E,p}(\sigma) \in \mathrm{GL}_2(\mathbb{F}_p)$ cannot be unipotent since its order is coprime to $p$. If $\rho_{E,p}(\sigma)\in\mathrm{GL}_2(\mathbb{Z}_p)$ has Jordan normal form $\left(
\begin{smallmatrix}
\lambda&1\\
0&\lambda \\
\end{smallmatrix}
\right)$, we have $\lambda = \frac{1}{2} \operatorname{tr} \rho_{E,p}(\sigma)\in\mathbb{Q}_p$. For $p>2$, the only roots of unity in $\mathbb{Q}_p$ are $1,\zeta_{p-1},\ldots,\zeta_{p-1}^{p-2}$, and they have pairwise distinct images in $\mathbb{F}_p^{\times}$ \cite[Ch. II, \S 5]{neukirch}. Since $\rho_{E,p}(\sigma)$ is not unipotent, it follows that $\lambda\not\equiv 1\pmod{p}$ so $\overline{\rho}_{E,p}(\sigma)$ cannot be unipotent.
\end{proof}

\begin{remark}
Proposition \ref{prop:p=2} below shows that Proposition \ref{prop:orderly-additive} is never true when $p=2$. The proposition may or may not hold when $p=3$. For example, suppose $E$ is the curve $y^2=x^3+x^2-18x-43$ with LMFDB-label $92.b 1$. Then $E$ has conductor $2^2 \cdot 23$, but $\overline{\rho}_{E,3}$ is reducible and has semi-simplification $\chi_{\mathrm{cyc},3} \oplus 1$ so there are no orderly primes for this representation. On the other hand, if $E$ is the curve $y^2=x^3-x^2-x$ with LMFDB-label $80.b 3$, $E$ has conductor $2^4 \cdot 5$. The image of $\overline{\rho}_{E,3}$ contains $\left(
\begin{smallmatrix}
2&2\\
0&2 \\
\end{smallmatrix}
\right)$ which is a non-unipotent element of $\mathrm{SL}_2(\mathbb{F}_3)$. By Lemma \ref{lem:orderly-condition}, the orderly primes for $\overline{\rho}_{E,3}$ have positive density.
\end{remark}

We now consider the case $p=2,3$ and prove following two propositions that complete the proof of Theorem \ref{thm:elliptic-orderly}.

\begin{proposition}\label{prop:p=2}
If $p =2$, then the set of orderly primes for $\overline{\rho}_{E,p}$ has positive density if and only if $E$ has no rational point of order $2$.
\end{proposition}

\begin{proof}
 If $E$ has a rational point order $2$, we can choose a basis for $E \left[ 2 \right]$ such that the image $\overline{\rho}_{E,2}$ is contained in $\left\langle \left(
\begin{smallmatrix}
1&1\\
0&1 \\
\end{smallmatrix}
\right) \right\rangle$, and hence there are no orderly primes for $\overline{\rho}_{E,2}$. Conversely, if $E$ has no rational point of order $2$, we see that the image of $\overline{\rho}_{E,2}$ must contain $\left(
\begin{smallmatrix}
0&1\\
1&1 \\
\end{smallmatrix}
\right)$ which is not unipotent so, by Lemma \ref{lem:orderly-condition}, we are done.
\end{proof}

\begin{proposition}
If $p=3$, then the set of orderly primes for $\overline{\rho}_{E,p}$ has positive density if and only if $E $ is not isogenous (over $\mathbb{Q}$) to a rational elliptic curve with a raitonal point of order $3$.
\end{proposition}

\begin{proof}
If $\overline{\rho}_{E,3}$ is irreducible, we have already seen that the set of orderly primes has positive density, so suppose that $\overline{\rho}_{E,3}$ is reducible. Using Lemma \ref{lem:orderly-condition}, we see that the set of orderly primes has positive density if and only if the image of $\overline{\rho}_{E,3}$ contains an element of the form $\left(
\begin{smallmatrix}
2&*\\
0&2 \\
\end{smallmatrix}
\right)$ which is equivalent to the semi-simplification of $\overline{\rho}_{E,3}$ not being isomorphic to $\chi_{\mathrm{cyc},3} \oplus 1$. This last condition exactly means that $E$ is not isogenous (over $\mathbb{Q}$) to a rational elliptic curve with a rational point of order $3$.
\end{proof}
\begin{remark}
An alternative characterization of when the orderly primes for the mod $p$ representation associated to an elliptic curve $E/\Q$ have positive density is that $E$ is not $\Q(\zeta_p)$-isogenous to an elliptic curve with $\Q(\zeta_p)$-rational $p$-torsion, see the introduction in \cite{katz81} and the exercise on page IV-6 in \cite{Serre68}. We would like to thank Ariel Weiss for pointing this out.  
\end{remark}
\subsubsection{The proof of Theorem \ref{thm:ordrlyellipticcurves}}

We now consider the case of two or more elliptic curves and prove Theorem \ref{thm:ordrlyellipticcurves}. Part (i) of the theorem follows from the following lemma: 

\begin{lemma} 
Let $\rho_1 , \rho_2 : G_{\mathbb{Q}} \rightarrow \mathrm{GL}_2(\mathbb{F}_p)$ be mod $p$ representations, each with determinant equal to the mod $p$ cyclotomic character. If the respective sets of orderly prime of order $1$ for $\rho_1$ and $\rho_2$ have positive density, then the set of orderly primes of order $1$ for $\rho_1 \oplus \rho_2$ also has positive density.
\end{lemma}

\begin{proof}
By Lemma \ref{lem:orderly-condition}, there are non-unipotent elements in $\rho_i(G_{\mathbb{Q}})\cap \mathrm{SL}_2(\mathbb{F}_p)$ for $i=1,2$. Hence Lemma \ref{lem:prod} implies that we can find $\tau \in G_{\mathbb{Q}}$ such that $\rho_i(\tau) \in \mathrm{SL}_2(\mathbb{F}_p)$, and $\rho_i(\tau)$ is not unipotent for each $i=1,2$. Therefore $(\rho_1 \oplus \rho_2)(\tau)$ does not have $1$ as an eigenvalue, and since $1 = \det \rho_1(\tau)= \chi_{\mathrm{cyc},p}(\tau)$, we also have $\tau \in \mathrm{Gal}(\overline{\mathbb{Q}} / \mathbb{Q}(\zeta_p))$. By the Chebotarev density theorem, it follows that the set of orderly primes for $\rho_1 \oplus \rho_2$ has positive density.
\end{proof}

\noindent
Using the lemma below, part (ii) and (iii) of Theorem \ref{thm:ordrlyellipticcurves} follow by the same argument as in the proof of Theorem \ref{thm:non-van}, by taking $c=3$ in the general case and $c=1$ in the semi-stable case.

\begin{lemma}
Let $E/ \mathbb{Q}$ be an elliptic curve, and $p$ a prime. Let $G \leq \mathrm{GL}_2(\mathbb{F}_p)$ denote the image of $\overline{\rho}_{E,p}$. If $p \geq 17$, and $(p, j_E ) \notin S$, then the proportion of unipotent elements in $G'$ is at most $3/(p-1)$. If $E$ is semi-stable, and $p \geq 11$, then this proportion is at most $1/(p-1)$.
\end{lemma}

\begin{proof}
We first introduce some notation. Any Cartan subgroup of $\mathrm{GL}_2(\mathbb{F}_p)$ is conjugate over $\mathrm{GL}_2(\mathbb{F}_{p^2})$ to either
\begin{equation*}
C_s := \left\{
  \begin{pmatrix}
a&0\\
0&b \\
\end{pmatrix}\,:\, a,b \in \mathbb{F}_p^\times \right\} \quad \textrm{or} \quad C_{ns}:= \left\{
\begin{pmatrix}
\alpha&0\\
0&\overline{\alpha} \\
\end{pmatrix}\,:\, \alpha \in \mathbb{F}_{p^2}^\times  \right\}
\end{equation*}
where $\overline{\alpha}$ denotes the non-trivial Galois conjugate of an element $\alpha \in \mathbb{F}_{p^2}$. A Cartan subgroup is conjugate to $C_s$ if is split and conjugate to $C_{ns}$ if it is non-split. The normaliser of a Cartan subgroup in $\mathrm{GL}_2(\mathbb{F}_p)$ is conjugate over $\mathrm{GL}_2(\mathbb{F}_{p^2})$ to either
\begin{equation*}
N_{s}:= \left\langle C_s ,
  \begin{pmatrix}
0&1\\
1&0 \\
\end{pmatrix} \right\rangle \quad \textrm{or} \quad N_{ns}:= \left\langle C_{ns},
\begin{pmatrix}
0&1\\
1&0 \\
\end{pmatrix} \right\rangle
\end{equation*}
according to it being split or non-split. We also encounter the following groups:
\begin{equation*}
L_s := \left\{
  \begin{pmatrix}
a&0\\
0&b \\
\end{pmatrix},\,
\begin{pmatrix}
0&a\\
b&0 \\
\end{pmatrix}\,:\, a,b \in \mathbb{F}_p^\times,\, a/b \in (\mathbb{F}_p^\times )^3 \right\} \leq  N_s
\end{equation*}
and 
\begin{equation*}
  L_{ns}:= \left\{g^3,\,
\begin{pmatrix}
0&1\\
1&0 \\
\end{pmatrix}g^3 \,:\, g \in C_{ns} \right\} \leq N_{ns}
\end{equation*}
Because of the relation
\begin{equation*}
\begin{pmatrix}
x&0\\
0&y \\
\end{pmatrix}
\begin{pmatrix}
0&1\\
1&0 \\
\end{pmatrix}
\begin{pmatrix}
x&0\\
0&y \\
\end{pmatrix}^{-1}
\begin{pmatrix}
0&1\\
1&0 \\
\end{pmatrix}^{-1} =
\begin{pmatrix}
y/x&0\\
0&x/y \\
\end{pmatrix}
\end{equation*}
we see that the derived subgroups of $N_s$ and $N_{ns}$ are
\begin{equation*}
N_s ' = \left\{
  \begin{pmatrix}
a&0\\
0&a^{-1} \\
\end{pmatrix}\,:\, a \in \mathbb{F}_p^\times  \right\} \quad \textrm{and}\quad N_{ns} ' = \left\{
\begin{pmatrix}
\alpha&0\\
0&\overline{\alpha} \\
\end{pmatrix}\,:\, \alpha \in \mathbb{F}_{p^2}^\times ,\,\alpha \overline{\alpha}=1 \right\},
\end{equation*}
and that the derived subgroups of $L_{s}$ and $L_{ns}$ are
\begin{equation*}
L_s ' = \left\{
  \begin{pmatrix}
a^3&0\\
0&a^{-3} \\
\end{pmatrix}\,:\, a \in \mathbb{F}_p^\times \right\} \quad \textrm{and} \quad L_{ns}' = \left\{
\begin{pmatrix}
\alpha^3&0\\
0&\alpha^{-3} \\
\end{pmatrix}\,:\, \alpha \in \mathbb{F}_{p^2},\, \alpha \overline{\alpha}=1 \right\}.
\end{equation*}
We now prove the first part of the lemma. If $E$ does not have complex multiplication, $p \geq 17$, and $(p, j_E) \notin S$, it follows by \cite[Proposition 1.13]{Zywina2015} that the image $G$ of $\rho_{E, p}$ is either all of $\mathrm{GL}_2(\mathbb{F}_p)$ or conjugate to $N_{ns}$ or $L_{ns}$. Hence the $G'$ is either $\mathrm{SL}_2(\mathbb{F}_p)$, $N_{ns}'$ or $L_{ns}'$. It is not hard to see that the proportion of unipotent elements in $\mathrm{SL}_2(\mathbb{F}_p)$ is
\begin{equation*}
\frac{p^2}{p(p-1)(p+1)} < \frac{1}{p-1}.
\end{equation*}
The groups $N_{ns}'$ and $L_{ns}'$ only consist of diagonal matrices so the identity matrix is the only unipotent element. Since $N_{ns}'$ is isomorphic to the kernel of the norm map from $\mathbb{F}_{p^2}$ to $\mathbb{F}_p$, one sees that $N_{ns}'$ is cyclic of order $p+1$. $ L_{ns}'$ is the subgroup of cubes in $N_{ns}'$ so it has index $1$ or $3$ in $N_{ns}'$. We conclude that the proportion of unipotent elements in $G'$ is at most $3/(p+1)$.\\

\noindent
If $E$ has complex multiplication, there will be two subcases depending on whether $j_E\neq 0$ or $j_E=0$. If $j_{E}\neq 0$, $p \geq 17$, and $(p,j_E)\notin S$, it follows by \cite[Proposition 1.14]{Zywina2015} that $G$ is conjugate to $N_s$ or $N_{ns}$ so $G'$ is conjugate to $N_s '$ or $N_{ns}'$. If $G$ is conjugate to $N_{ns}$, we have already seen that the proportion of unipotent elements in $G'$ is $1/(p+1)$. If $G$ is conjugate to $N_s$, $G'$ has order $p-1$ and is conjugate to a group consisting of diagonal matrices, so the proportion of unipotent elements in $G'$ is $1/(p-1)$.

If $j_{E}=0$, and $p \geq 17$, it follows by \cite[Proposition 1.16]{Zywina2015}, that $G$ is conjugate to $N_{s}$, $N_{ns}$, $L_{s}$ or $L_{ns}$. The only case that we have not already considered is when $G$ is conjugate to $L_{s}$. The group $L_s '$ only consist of diagonal matrices and is isomorphic to the group of cubes in $\mathbb{F}_p^\times$ so the proportion of unipotent elements in $G'$ is either $1/(p-1)$ or $3/(p-1)$.\\

\noindent
The last part of the lemma follows from a theorem of Mazur \cite[Theorem 4]{mazur78} which states that $\overline{\rho}_{E,p}$ is surjective for all $p \geq 11$ if $E$ is semi-stable. Under these conditions, the derived subgroup of $G$ is $\mathrm{SL}_2(\mathbb{F}_p)$, and we have already seen that the proportion of unipotent elements in this group is at most $1/(p-1)$. 
\end{proof}

\subsection{Orderly primes in higher rank}
\label{sec:gln}
When $n \geq 3$, it is no longer true that absolute irreducibility of $\rho: G_{\mathbb{Q}}\rightarrow \mathrm{GL}_n(\mathbb{F}_q)$ implies the existence of infinitely many orderly primes for $\rho$. For example, if $n\geq 3$ is odd, the group $\mathrm{SO}_n(\mathbb{F}_q)$ acts (absolutely) irreducibly on $\mathbb{F}_q^n$ when $q$ is odd, but every element of $\mathrm{SO}_n(\mathbb{F}_q)$ has $1$ as an eigenvalue. When we impose additional assumptions on the image of $\rho$, we can still find infinitely many orderly primes for $\rho$. 

\begin{proposition}
Let $n \geq 2$, and suppose $\rho: G_{\mathbb{Q}} \rightarrow \mathrm{GL}_n(\mathbb{F}_q)$ is a Galois representation with image $G$. If $G \supset \mathrm{SL}_n(\mathbb{F}_q)$, then the set of orderly primes of order $m$ for $\rho$ has positive density for all $m \geq 1$.
\end{proposition}

\begin{proof}
With the exceptions of $(n,q)=(2,2),(2,3)$, the group $\mathrm{SL}_n(\mathbb{F}_q)$ is perfect \cite{special}. Suppose we are not in the cases $(n,q)=(2,2),(2,3)$. Then $G'= \mathrm{SL}_n(\mathbb{F}_q)$, and we can certainly find an element of $\mathrm{SL}_n(\mathbb{F}_q)$ not having $1$ as an eigenvalue. Indeed, write $n=2a+3b$ where $a,b \geq 0$ are integers, and let $g \in \mathrm{SL}_n(\mathbb{F}_q)$ be the block diagonal matrix where $a$ blocks are equal to
\begin{equation*}
\begin{pmatrix}
0&1\\
-1&1 \\
\end{pmatrix},
\end{equation*}
and $b$ blocks are equal to
\begin{equation*}
\begin{pmatrix}
0&1&0\\
0&0&1\\
1&0&1  
\end{pmatrix}.
\end{equation*}
These matrices have determinant $1$ and do not have $1$ as an eigenvalue so the same is true for $g$. We conclude using Proposition \ref{first_criteria}.

We now consider the special cases $(n,q)=(2,2),(2,3)$. By Proposition \ref{first_criteria}, it is enough to show that $\mathrm{SL}_2(\mathbb{F}_q)'$ contians a non-unipotent element, or equivalently that it contains an element whose order is not a power of $p$. Have $\mathrm{SL}_2(\mathbb{F}_2)\cong S_3$ so $S_3 ' \cong \mathbb{Z} / 3 \mathbb{Z}$, and $G'$ contains an element of order $3$. The derived subgroup of $\mathrm{SL}_2(\mathbb{F}_3)$ is isomorphic to the quaternion group $Q_8$ so $G'$ contains an element of order $4$.  
\end{proof}

In the case of even rank, we can get away with a slightly weaker assumption on $\rho$.

\begin{proposition}
Let $n \geq 2$, and suppose $\rho: G_{\mathbb{Q}}\rightarrow\mathrm{GL}_{2n}(\mathbb{F}_q)$ is a Galois representation with image $G$. If $G \supset \mathrm{Sp}_{2n}(\mathbb{F}_q)$, then the set of orderly primes of order $m$ has positive density for all $m \geq 1$.
\end{proposition}

\begin{proof}
Unless $(n,q)=(2,2)$, $\mathrm{Sp}_{2n}(\mathbb{F}_q)$ is perfect \cite{symplectic}. Assume we are not in the case $(n,q)=(2,2)$. Then $G'\supset \mathrm{Sp}_{2n}(\mathbb{F}_q)$, and it is enough to write down an element of $\mathrm{Sp}_{2n}(\mathbb{F}_q)$ which does not have $1$ as an eigenvalue. Choose an element $h \in \mathrm{GL}_n(\mathbb{F}_q)$ not having $1$ as an eigenvalue (this can for example be achieved by using the matrices in the proof of the previous proposition). Then the block matrix $g = \left(
\begin{smallmatrix}
h&0\\
0&h^{-t} \\
\end{smallmatrix}
\right) \in \mathrm{Sp}_{2n}(\mathbb{F}_q)$ has the desired property (here $h^{-t}$ is the inverse of the transpose of $h$).

In the special case $(n,q)=(2,2)$, there is an exceptional isomorphism $\mathrm{Sp}_{4}(\mathbb{F}_2) \cong S_6 $, and hence $\mathrm{Sp}_4(\mathbb{F}_2 )' \cong A_6$. Let $h = \left(
\begin{smallmatrix}
0&1\\
1&1 \\
\end{smallmatrix}
\right)$, and $g = \left(
\begin{smallmatrix}
h&0\\
0&h^{-t} \\
\end{smallmatrix}
\right) \in \mathrm{Sp}_4(\mathbb{F}_2)$. Since $h$ has order $3$, $g$ also has order $3$. Hence $g$ must lie in $\mathrm{Sp}_4(\mathbb{F}_2) '$ since no element in $S_6 \setminus A_6$ has order $3$. Since $h$ does not have $1$ as an eigenvalue, the same is true for $g$. By Proposition \ref{first_criteria}, the proof is complete.
\end{proof}

\section{Character zeroes of $p$-adic measures}
In this section we will prove some variations and refinements of the results obtained in \cite[Section 2]{KrizNordentoft23} regarding the character zeroes of horizontal $p$-adic measures. In the next section we will argue that there exist interesting such horizontal $p$-adic measures, horizontal $p$-adic $L$-functions, interpolating the family of $L$-values that we care about. In Section \ref{sec:prop} we will use all of these ideas to obtain a new propagation of non-vanishing result ensuring non-vanishing for \emph{all} Galois characters of certain abelian extensions of $\Q$.
\subsection{Horizontal $p$-adic measures} Fix throughout a prime number $p$. We will now recall the setting. We say that an abelian profinite $(G,+)$ is \emph{horizontal pro-$p$} if it is isomorphic to an infinite product of finite cyclic $p$-groups, i.e.\ 
$$G\cong \prod_{n\in \N} \Z/p^{m_n},$$ 
with $(m_n)_{n\in \N}\subset \Z_{\geq 1}$ where the product is viewed as a profinite group via the product topology. We define the \emph{exponent of $G$} as 
\begin{equation}
    e(G):= \sup_{g\in G} \log_p(|g|), 
\end{equation}
where $|g|:=|\langle g\rangle|\in \N\cup\{\infty\}$ denotes the order of an element $g\in G$ and $\log_p$ denotes the base-$p$ logarithm. Since $G$ is pro-$p$ the power of any element is indeed either a power of $p$ or $\infty$ so that $e(G)\in \N\cup \{\infty\}$. We say that $G$ has \emph{bounded exponent} if $e(G)<\infty$. For $G$  a horizontal pro-$p$ abelian group this means exactly that the sequence $(m_n)_{n\in \N}$ is bounded in which case $e(G)=\max_{n\in \N} m_n$. 

Let $\C_p$ denote the completion of an algebraic closure of $\Q_p$, let $|\!\cdot\!|_p: \C_p\rightarrow \R_{\geq 0}$ denote the $p$-adic norm normalized so that $|p|_p=p^{-1}$ and let $\Oo_{\C_p}=\{x\in \C_p\mid |x|_p\leq 1\}$ denote the valuation ring.  Let $R\subset \Oo_{\C_p}$ be a $p$-adically complete subring and define the \emph{Iwasawa algebra} associated with $G$ as:
$$ R\llbracket G\rrbracket:=\varprojlim_{H\leq G\text{ open}} R[G/H].$$
One can identify the Iwasawa algebra of $G$ with integral $p$-adic measures on $G$, in the sense that there is an $R$-algebra isomorphism:
\begin{equation}\label{eq:measure}R\llbracket G\rrbracket\cong \mathrm{Hom}_\mathrm{cts}(\mathcal{C}(G,R),R),\end{equation}
where the product structure on the right hand side is given by convolution. Here $\mathcal{C}(G,R)$ denotes the $R$-module of continuous function from $G$ to $R$ equipped with the sup-norm topology, see \cite[Section 2.1]{JacintoWilliams25} for details. Via the inclusion $R\llbracket G\rrbracket\subset \Oo_{\C_p}\llbracket G\rrbracket $ we get an $R$-bilinear pairing 
$$ R\llbracket G\rrbracket \times \mathcal{C}(G,\Oo_{\C_p})\rightarrow \Oo_{\C_p}, \quad (\nu,\varphi)\mapsto \nu(\varphi).$$
In particular, we obtain the Fourier transform of a measure $\nu\in R\llbracket G\rrbracket$: 
$$ \hat{\nu}:\widehat{G}\rightarrow \Oo_{\C_p},\quad \chi\mapsto\nu(\chi), $$
where $\widehat{G}$ denotes the group of continuous $\mathcal{O}_{\C_p}^\times$-valued characters of $G$. We will denote the trivial character by 
$$\1: G\rightarrow \Oo_{\C_p}^\times,\quad g\mapsto 1. $$
The principal goal in the theory developed in \cite[Section 2]{KrizNordentoft23} is to understand (interpreted in an appropriate sense) the image of $\hat{\nu}$. An example is to understand the set of \emph{character zeroes}, i.e.\ $\chi\in \widehat{G}$ such that $\nu(\chi)=0$. But we can also aim to understand more delicate properties of the structure of the image. 

With this in mind, denote the \emph{sup-norm of $\hat{\nu}$} by
\begin{equation}
    \lVert \hat{\nu} \rVert_p:=\sup_{\chi\in \widehat{G}} |\nu(\chi)|_p.
\end{equation}
Note that if $R$ has discrete $p$-adic valuation and $G$ has bounded exponent then the supremum is attained. We end this section by stating a key result from \cite[Section 2]{KrizNordentoft23}, which shows that the maximum is attained for ``a positive proportion'' of characters.

\begin{theorem}[Theorem 2.11 in \cite{KrizNordentoft23}] \label{thm:structure}Let $(G,+)$ be a horizontal pro-$p$ abelian group of bounded exponent and let $R\subset \Oo_{\C_p}$ be a complete subring with discrete $p$-adic valuation.
    Let $\nu \in R\llbracket G \rrbracket$ be a non-zero $p$-adic measure.  Then there exists a finite subgroup $M_\nu\leq \widehat{G}$ such that the following holds: for any non-trivial finite subgroup $M\leq \widehat{G}$ there exists $\chi\in M\setminus M^p$ and $\chi_0\in M_\nu$ such that 
    \begin{equation}
        |\nu(\chi\chi_0)|_p=\lVert \hat{\nu}\rVert_p. 
    \end{equation}
\end{theorem}
\subsection{Non-vanishing orbits of characters}
For applications to Diophantine rank stability we will need a result for positive proportion non-vanishing for full orbits of characters. The result follows from Theorem \ref{thm:structure} after some purely measure theoretic constructions.
\begin{theorem}\label{thm:structurenew}
    Let $(G,+)$ be a horizontal pro-$p$ abelian group of bounded exponent and let $R\subset \Oo_{\C_p}$ be a complete subring with discrete $p$-adic valuation.
    Let $\nu \in R\llbracket G \rrbracket$ be a $p$-adic measure satisfying $\nu(\1)\neq 0$.  Then there exists a finite subgroup of characters $M_\nu\leq \widehat{G}$ such that the following holds: for any character $\chi\in \widehat{G}$ there exists  $\chi_0\in M_\nu$ such that for all $n\in \Z$ 
    \begin{equation}
        \nu((\chi\chi_0)^n)\neq 0 . 
    \end{equation}
\end{theorem}
\begin{proof}
For every $n\in \Z$ we obtain a continuous group homomorphism 
$$p_n:G\rightarrow G,\quad g\mapsto n\cdot g.$$
It follows from the description (\ref{eq:measure}) and standard measure theory that we get from $p_n$ an induced pushforward map of spaces of $p$-adic measures
$$(p_n)_\ast:R\llbracket G\rrbracket \rightarrow R\llbracket G\rrbracket,\quad \nu \mapsto \nu^{(n)}:=(p_n)_\ast(\nu),$$
uniquely characterized by the interpolation property
$$\nu^{(n)}(\chi)= \nu(\chi^n)\quad \text{for all }\chi\in \widehat{G}. $$
Let $N=e(G)$ denote the exponent of $G$ and consider the following ``norm'' of $\nu$:
\begin{equation}\label{eq:normofmeasure} \nu^\mathrm{Nr}:= \prod_{n=1}^{p^N} \nu^{(n)}\in  R\llbracket G\rrbracket.\end{equation}
Since $ \nu^\mathrm{Nr}(\1)=\nu(\1)^{p^N}\neq 0$, this is a non-zero measure and so we can apply Theorem \ref{thm:structure}. Let $\tilde{M}_\nu\leq \widehat{G}$ be the finite subgroup obtained from Theorem \ref{thm:structure}. We claim that we can take $M_\nu=\tilde{M}_\nu$ in the theorem. Note that for $\chi=\1$ the statement is trivially true. So let $\chi\in \widehat{G}$ be a non-trivial character and consider the finite subgroup $M=\langle \chi\rangle$. We observe that the order of $\chi$  divides $p^N$. Then we conclude the existence of $\chi'\in M\setminus M^p$ and $\chi_0\in \tilde{M}_\nu$ such that
\begin{equation}
    \label{eq:normineq}\prod_{n=1}^{p^N} |\nu((\chi'\chi_0)^n)|_p=|\nu^\mathrm{Nr}(\chi'\chi_0)|_p=\lVert \widehat{\nu^\mathrm{Nr}}\rVert_p\geq |\nu^\mathrm{Nr}(\1)|_p=\left(|\nu(\1)|_p\right)^{p^N}. \end{equation} 
Notice that the condition $\chi'\in \langle \chi\rangle\setminus \langle \chi^p\rangle  $ means exactly that $\chi'=\chi^j$ for some $j\in \Z$ coprime to $p$. This implies that $\chi\mapsto \chi^j$ defines an automorphism of $\tilde{M}_\nu$ and so we can write
$$ \chi'\chi_0=(\chi \tilde{\chi}_0)^j,  $$
for some $\tilde{\chi}_0\in \tilde{M}_\nu$ and $\chi$ being the non-trivial character in question. Note now that since $(j,p)=1$ and  the order of $\chi \tilde{\chi}_0$ divides $p^N$ we have the following equality:
$$\{\langle (\chi \tilde{\chi}_0)^n\mid 1\leq n\leq p^N\}= \langle \chi \tilde{\chi}_0 \rangle =\langle(\chi \tilde{\chi}_0)^j\rangle=\{ (\chi \tilde{\chi}_0)^{j\cdot n}\mid 1\leq n\leq p^N\}. $$
Combining this with the non-vanishing of the left-hand side of (\ref{eq:normineq}) we get the desired result.  
\end{proof}
Using the same ideas also allows to remove the linearly disjointness assumption (i.e.\ $R[1/p]\cap \Q_p(\mu_p)=\Q_p$) from the last part of \cite[Theorem 2.17]{KrizNordentoft23}.
\begin{corollary}

Let $(G,+)$ be a horizontal pro-$p$ abelian group of bounded exponent and let $R\subset \Oo_{\C_p}$ be a complete subring with discrete $p$-adic valuation.
    Then for $\nu \in R\llbracket G \rrbracket$ there exists a finite subgroup $M_\nu\leq \widehat{G}$ such that the following holds: for any character $\chi\in \widehat{G}$ there exists $\chi_0\in M_\nu$ such that 
    \begin{equation}
        |\nu(\chi\chi_0)|_p=\lVert \hat{\nu}\rVert_p. 
    \end{equation}    
\end{corollary}

\begin{proof}
By $p$-adic discreteness the supremum defining $\lVert \hat{\nu}\rVert_p$ is attained for some character $\chi^\ast\in\widehat{G}$. By considering the twist $\nu^\ast$ of $\nu$ by $\chi^\ast$ as in \cite[Section 2.2.1]{KrizNordentoft23} we may assume that $\chi^\ast=\1$ since then we can pick $M_\nu=\langle \chi^\ast,M_{\nu^\ast}\rangle $. Applying the argument from (\ref{eq:normineq}) in the proof of the previous theorem we conclude that 
$$\prod_{n=1}^{p^N} |\nu((\chi\chi_0)^n)|_p\geq \left(|\nu(\1)|_p\right)^{p^N}.$$
Since by assumption 
$$|\nu((\chi\chi_0)^n)|_p \leq |\nu(\1)|_p,$$
for all $n\in \Z$ we conclude, in particular, that  $$|\nu(\chi\chi_0)|_p =|\nu(\1)|_p=\lVert \hat{\nu}\rVert_p,$$
which gives the desired conclusion.
\end{proof}
\begin{remark}
    In the case of general coefficients $\nu\in \Oo_{\C_p}\llbracket G\rrbracket $  (and $G$ horizontal pro-$p$ abelian of bounded exponent) the same argument yields the following variation (cf.\ \cite[Thm.\ 2.12]{KrizNordentoft23}): For any $\eps>0$ there exists a finite subgroup $M_{\nu,\eps}\leq \widehat{G}$ such that the following holds: for any $\chi\in \widehat{G}$ there exists $\chi_0\in M_{\nu,\eps}$ such that
    $$ |\nu(\chi\chi_0)|_p>\lVert \hat{\nu}\rVert_p-\eps. $$
\end{remark}
\section{Horizontal $p$-adic $L$-functions via  Kloosterman twist $L$-series}\label{sec:horPadic}
In this section, we show how the condition of orderly primes naturally appears in the context of horizontal norm relations in arbitrary rank. From this we will extract condition for the existence of \emph{horizontal $p$-adic $L$-functions} which are certain interesting elements of the horizontal Iwasawa algebras discussed in the previous section. As such this section serves a twofold purpose; it motivates the $\GL_n$-results of the Section \ref{sec:orderly}, as well as serving as a future reference.

We will try to isolate the relevant properties of the $L$-functions we will be interested in\footnote{This is supposed to model the case where $L(M,s)$ is the $L$-function associated to a motive $M$ over $\Q$ with a critical value in the sense of \cite{Deligne79}, e.g.\ a (classical) holomorphic newform of even weight $k$.}: consider a meromorphic function $L(M,s)$ of $s\in \C$  satisfying the following: 
\begin{enumerate}
    \item (Euler product) There is a Dirichlet series with an Euler product of degree $d\geq 1$
    \begin{align*}
        \nonumber L(M,s)=\sum_{n\geq 1}\frac{\lambda_M(n)}{n^s}&=\prod_p (1-\alpha_M(p, 1)p^{-s})^{-1} \cdots (1-\alpha_M(p, d)p^{-s})^{-1}
         =\prod_p P_{M,p}(p^{-s})^{-1},
    \end{align*}
    converging absolutely in some half-plane $\Re s>\sigma_0$. We refer to $P_{M,p}(X)\in \C[X]$ as the \emph{Hecke polynomial  for $M$ at $p$}.  
    \item (Finite conductor) An integer $q_M\geq 1$, called the \emph{conductor} of $M$ such that $\alpha_M(p, i)\neq 0$ for all $1 \leq i \leq d$ exactly if $p\nmid  q_M$, or equivalently $P_{M,p}(X)$ is of degree $d$ if and only if $p\nmid  q_M$.
    \item (Algebraicity) There exists a number field $K\subset \C$, the \emph{Hecke field of $M$}, such that $P_{M,p}(X)\in \mathcal{O}_K[\tfrac{1}{p}][X]$ for all primes $p$.
    \end{enumerate}

The basis of the construction of (horizontal) $p$-adic $L$-functions are the arithmetic properties of certain associated \emph{periods}. Consider the hyper-Kloosterman sum
$$\Kl_m(a;q):= \sum_{\substack{x_1,\ldots, x_m\modulo q\\x_1\cdots x_m=1 }} e\left(\frac{x_1+\ldots + x_{m-1}+ax_m}{q}\right), $$
defined for $q,m\geq 1$ and $a\in \Z/q$. We consider the \emph{twisted $L$-series}:  
\begin{equation}\label{eq:kloosttwist}L(M, \tfrac{a}{q}, m, s):=\sum_{\substack{n\geq 1,\\ (n,q)=1}} \frac{\lambda_M(n) \Kl_m(an;q)}{n^s}, \qquad \Re s>\sigma_0.\end{equation}
In Section \ref{sec:higherrank} we will see examples of these $L$-series with good properties for general $m$ when $M$ corresponds to a direct sum of newforms of even weight.  In general, we need the following properties\footnote{In practice, these properties are verified by studying \emph{automorphic periods} of the (conjectural) automorphic representation associated to $M$.}. 

    \begin{enumerate}
    \item[(4)] (Existence of periods) There exists an integer $m\geq 1$ such that the Dirichlet series $L(M, \tfrac{a}{q}, m, s)$ admits meromorphic continuation to $\C$ for all square-free  $q\geq 1$ coprime to $q_M$ and  $a\in \Z/q$.
    \item[(5)] (Criticality at $s=0$) There exist two complex numbers $\Omega^+,\Omega^- \in \C^\times$ such that for all $\frac{a}{q}\in \Q_{>0}$ with $(q,q_M)=1$ and a choice of sign $\pm$ it holds that
    \begin{equation}\nonumber
        \frac{1}{2\Omega^\pm }\left( L(M, \tfrac{a}{q}, m, 0)\pm L(M, \tfrac{-a}{q}, m, 0) \right)\in \mathcal{O}_K[\tfrac{1}{q}].
    \end{equation} 
    
\end{enumerate} 
This last criticality condition is modeled on critical values for motivic $L$-functions  (cf.\ \cite{Deligne79}). Note also that for $m=1$ the twisted $L$-series equals the \emph{additive twist $L$-series} (cf.\  \cite[Section 3.3]{Nordentoft20.4}). Define for a Dirichlet character $\chi\modulo q_\chi$ and a positive integer $q$ the (partial) twisted $L$-function:
$$L^{(q)}(M\otimes \chi, s):= \sum_{\substack{n\geq 1,\\ (n,q)=1}}\frac{\lambda_M(n)\chi(n)}{n^s},\quad \Re s>\sigma_0,$$
and write simply $L(M\otimes \chi, s)=L^{(1)}(M\otimes \chi, s)$. Note that the partial $L$-function admits an Euler product of conductor $q\cdot q_M$ with Hecke polynomial at $p\nmid q\cdot q_\chi$ equal to $P_{M,p}(\chi(p)X)$.

The key to the construction of horizontal $p$-adic $L$-functions are the horizontal norm relations. These are related to the following Birch--Stevens type formula, cf.\ \cite[Section 8]{MazurTateTeitelbaum}, for \emph{non-primitive} Dirichlet characters.
\begin{lemma}[Non-primitive Birch--Stevens formula]\label{lem:npBS}
Let $\chi$ be a Dirichlet character modulo $q$. Let $\chi^\ast\modulo q^\ast$ with $q^\ast|q$ be the primitive character that induces $\chi$. 
Then it holds for $\Re s>\sigma_0$ that
\begin{equation}\label{eq:npBS} \sum_{a\modulo q}\chi(a)L(M, \tfrac{a}{q}, m,s)=(\mu(q/q^\ast)\chi^\ast(q/q^\ast))^m \tau(\chi^\ast)^m L^{(q)}(M\otimes \overline{\chi}^\ast,s).   \end{equation}
\end{lemma}
\begin{proof}
By absolute convergence we can for $\Re s>\sigma_0$ interchange the sum over $a\modulo q$ and that over $n\geq 1$ in the Dirichlet series for $L(M, \tfrac{a}{q}, m,s)$. This way the innermost sum becomes 
\begin{align*} 
  \sum_{a\modulo q}\chi(a) K_m(an;q)=\sum_{\substack{x_1,\ldots, x_m\modulo q\\x_1\cdots x_m=1 }} e\left(\frac{x_1+\ldots + x_{m-1}}{q}\right)\sum_{a\modulo q}\chi(a) e\left(\frac{anx_m}{q}\right). \end{align*}
  Now using that $(n x_m,q)=1$ and the formula \cite[Lemma 3]{Sh75} for non-primitive Gauss sums we have 

\begin{align*}
    \sum_{a\modulo q}\chi(a) e\left(\frac{anx_m}{q}\right)&=\tau(\chi^\ast)\sum_{d|(nx_m,q/q^\ast)} d\, \overline{\chi}^\ast\left(\frac{nx_m}{d}\right) \mu\left(\frac{q}{d q^\ast}\right)\chi^\ast\left(\frac{q}{dq^\ast}\right)\\
    &=\tau(\chi^\ast)\overline{\chi}^\ast(n x_m)\mu\left(\frac{q}{q^\ast}\right)\chi^\ast\left(\frac{q}{q^\ast}\right). 
\end{align*}
Now we insert this expression into (\ref{eq:npBS}), use that $\overline{\chi}^\ast(x_m)=\chi^\ast(x_1)\cdots \chi^\ast(x_{m-1})$ and evaluate the remaining sum over $x_1,\ldots, x_{m-1}\modulo q $ as
$$ \left(\sum_{x\modulo q} \chi(x)e\left(\frac{x}{q}\right)\right)^{m-1}=(\mu(q/q^\ast)\chi^\ast(q/q^\ast)\tau(\chi^\ast))^{m-1},  $$
using again \cite[Lemma 3]{Sh75}. Inserting this into (\ref{eq:npBS}) yields the desired equality. 
\end{proof}
In particular, we conclude that the assumptions (1)-(4) imply that the partial twisted $L$-function $L^{(q)}(M\otimes \chi,s)$ admits meromorphic continuation to $\C$ satisfying (1)-(3). 

For every integer $q\geq 2$ we put  $R_q=\mathcal{O}_K[\frac{1}{q}]$. Then we can reinterpret the above in terms of group algebras (or measures) via the \emph{theta elements}
\begin{equation}
 \theta^\pm_q:= \sum_{a\in (\Z/q)^\times} \left( \frac{1}{2\Omega^\pm }\left( L^{(q)}(M, \tfrac{a}{q}, m,0)\pm L^{(q)}(M, \tfrac{-a}{q}, m,0)\right)\right) [a]\in R_q[(\Z/q)^\times],  
\end{equation}
for $q\geq 2$. For $q|q'$ we have canonical projections
$$ \pi_{q',q}: (\Z/q')^\times\rightarrow (\Z/q)^\times,  $$
and push-forward along this map induces a ring homomorphism
$$(\pi_{q',q})_\ast: R_{q'}[(\Z/q')^\times]\rightarrow R_{q'}[(\Z/q)^\times]. $$

\noindent
The horizontal norm relations are concerned with the effect of these maps on theta elements.
\begin{corollary}[Horizontal norm relations] 
Let $q\geq 2$ be square-free and $\ell$ a prime not dividing $q$. Then it holds that
\begin{equation}\label{eq:hornormrel}
   (\pi_{q\ell,q})_\ast(\theta^\pm_{q\ell} )=(-[\ell])^m P_{M,\ell}([\ell]^{-1})\, \theta^\pm_{q}\in R_{q\ell}[(\Z/q)^\times], 
\end{equation}  
where $[\ell]\in R_{q\ell}[(\Z/q)^\times]$ denotes the basis element corresponding to $\ell\modulo q$. 
\end{corollary}
\begin{proof}
We start by noting that for a Dirichlet character $\chi \modulo q$ it holds that
$$L^{(q\ell)}(M\otimes\overline{\chi},0)=P_{M,\ell}(\overline{\chi}(\ell))L^{(q)}(M\otimes\overline{\chi},0). $$
Thus it follows from Lemma \ref{lem:npBS} that for $\chi$ a character of $(\Z/q)^\times$ the left- and right-hand side of (\ref{eq:hornormrel}) agree. Note here that $(\pi_{q\ell,q})_\ast(\theta^\pm_{q\ell} )(\chi)=\theta^\pm_{q\ell}(\tilde{\chi})$ where $\tilde{\chi}\modulo q\ell$ denotes the character induced from $\chi\modulo q$.  This yields the wanted equality of elements of $R_{q\ell}[(\Z/q)^\times]$ by Fourier inversion.   
\end{proof}
Fix now a prime $p$ and denote by $\overline{\Q}_p$ an algebraic closure of $\Q_p$. We will  through-out fix an embedding \begin{equation}
\label{eq:embed} \overline{\Q}\subset \overline{\Q}_p, 
\end{equation} of the algebraic closure $\overline{\Q}\subset \C$  of $\Q$. This determines a prime ideal $\overline{\mathfrak{p}}$ of $\overline{\Z}$ above $p$ given by the intersection of $\overline{\Z}$ with the maximal ideal of $\Oo_{\C_p}$. The intersection  $\mathfrak{p}:=\overline{\mathfrak{p}}\cap K$ yields a prime of the Hecke field $K$ above $p$. We will now restrict the above discussion  to the case  where $q=\ell_1\cdots \ell_n$ and $\ell=\ell_{n+1}$ with $\ell_i$  distinct primes congruent to $1$ modulo $p$. In this case we obtain embeddings $\Oo_K[\frac{1}{q\ell}]\subset \Oo_{K_\mathfrak{p}}\subset \overline{\Q}_p$. Fix  compatible projections
\begin{equation}\label{eq:proj}(\Z/q)^\times \twoheadrightarrow \prod_{i=1}^n \Z/p^{m_i},\qquad (\Z/q\ell)^\times \twoheadrightarrow \prod_{i=1}^{n+1} \Z/p^{m_i},\end{equation}
where $m_i=v_p(\ell_i-1)\geq 1$ (e.g.\ pick primitive roots modulo $\ell_i$ for $i=1,\ldots, n+1$). By push-forward along these projections we obtain from (\ref{eq:hornormrel}) the equality 
\begin{equation}\label{eq:hornormrelmodp}
   (\tilde{\pi}_{n})_\ast(\tilde{\theta}^\pm_{q\ell} )=(-1)^m[m\sigma_\ell]\,  P_{M,\ell}([-\sigma_\ell])\, \tilde{\theta}^\pm_{q}\in \Oo_{K_\mathfrak{p}}[\Pi_{i=1}^n \Z/p^{m_i}], 
\end{equation} 
where $\tilde{\pi}_n:\prod_{i=1}^{n+1} \Z/p^{m_i}\rightarrow \prod_{i=1}^{n} \Z/p^{m_i} $ is the natural projection and  $\sigma_\ell\in\prod_{i=1}^{n} \Z/p^{m_i} $ is the image of $\ell\modulo \ell_1\cdots \ell_n$ under the first projection map in (\ref{eq:proj}), cf.\ \cite[eq.\ (5.3)-(5.6)]{KrizNordentoft23}. The significance of the notion of orderly primes in Definition \ref{def:orderly} is illustrated by the following lemma, keeping assumptions and notation from above.
\begin{lemma}\label{lem:fudgeinv}
Assume that $P_{M,\ell}(1)\not\equiv 0 \modulo \mathfrak{p}$. Then
\begin{equation}\label{eq:element} (-1)^m[m\sigma_\ell] P_{M,\ell}([-\sigma_\ell]),\end{equation}
is invertible as an element of $\mathcal{O}_{K_\mathfrak{p}}[\prod_{i=1}^n \Z/p^{m_i}]$. 
\end{lemma}
\begin{proof}
When evaluating the element at the trivial  character $\mathbf{1}: \prod_{i=1}^n \Z/p^{m_i}\rightarrow \C_p^\times$ we obtain
\begin{align*}(-1)^m[m\sigma_\ell]P_{M,\ell}([-\sigma_\ell])(\mathbf{1})= (-1)^m P_{M,\ell}(1),\end{align*}
which is non-zero modulo $\mathfrak{p}$ by the assumption. This means that the left-hand side above is an element of $\left(\mathcal{O}_{K_\mathfrak{p}}\right)^\times$ which implies by \cite[Proposition 2.2]{KrizNordentoft23} that the element (\ref{eq:element}) is indeed invertible.
\end{proof}
Finally, to link it up with the previous section we assume that there is a ``residual  realization of $M$'' having a positive proportion of orderly primes:
\begin{enumerate}
    \item[(6)] (Galois representation) There exists a Galois representation $\rho: G_\Q\rightarrow \GL_d(\F_q) $  with $q=| \mathcal{O}_{K}/\mathfrak{p}|$ which is unramified outside of $p\cdot  q_M$ and such that for all primes $\ell$ not dividing $ p\cdot q_M$
$$P_{M,\ell}(X)\modulo \mathfrak{p}= \det\left(I- X\rho( \Frob_\ell) \right).$$
\item[(7)] (Orderly) The set of orderly primes (of order $1$) for $\rho$ has positive density.
\end{enumerate} 
Under these assumptions we can associated a \emph{horizontal $p$-adic $L$-function} in the sense of \cite[Definition 5.3]{KrizNordentoft23}, i.e.\ an element of the Iwasawa algebra (\ref{eq:measure}), meaning a horizontal $p$-adic measure, interpolating $L$-values twisted by characters of $p$-power order. 
Retaining the notation from above we arrive at the following general construction.
\begin{theorem}\label{thm:padicL}
Let $L(M,s)$ be a meromorphic function in $s\in \C$ satisfying the conditions (1)-(7). Let $(\ell_n)_{n\in \N}$ be a sequence consisting of distinct orderly primes for $\rho$, put $m_n=v_p(\ell_n-1)\geq 1,n\in \N$ and fix a component-wise projection $\rho: \prod_{n\in \N}(\Z/\ell_n)^\times \twoheadrightarrow \prod_{n\in \N} \Z/p^{m_n}$. Then for each choice of sign $\pm$ there exists a unique $p$-adic measure, the \emph{horizontal $p$-adic $L$-function associated to $M$},  
\begin{align}
    \nu^\pm_{M,p}\in \mathcal{O}_{K_\mathfrak{p}} \llbracket \prod_{n\in \N} \Z/p^{m_n}
    \rrbracket, 
\end{align}
satisfying the following interpolation property: Let $\chi\modulo D$ be a primitive Dirichlet character satisfying $\chi(-1)=\pm 1$ of $p$-power order and conductor $D$ given by a product of distinct orderly primes  among $(\ell_n)_{n\in \N}$. Then it holds that
\begin{align}\label{eq:interpolation}
  \nu^\pm_{M,p}(\tilde{\chi})= \prod_{\ell| D}\left((-\chi^{(\ell)} (\ell))^m  P_{M,\ell}(\overline{\chi}^{(\ell)}(\ell))^{-1}\right) \frac{\tau(\chi)^mL(M\otimes \overline{\chi},0)} {\Omega^\pm},  
\end{align}
 where $\tilde{\chi}:\prod_{n\in \N} \Z/p^{m_n}\rightarrow \overline{\Q}^\times\subset \overline{\Q}_p^\times $ is the unique continuous character so that $\chi=\tilde{\chi}\circ \rho$ and $\chi^{(\ell)}$ denotes the unique Dirichlet character modulo $\frac{D}{\ell}$ given by restriction via the splitting  $(\Z/\frac{D}{\ell})^\times \hookrightarrow (\Z/D)^\times$.
\end{theorem}
\begin{proof}
Given the above, the construction of the horizontal $p$-adic $L$-function is rather formal. We will simply sketch the argument and refer to \cite[Section 5.1]{KrizNordentoft23} for further details. 

The first step is to define a ``global lift'' of the fudge factors in (\ref{eq:element}). More precisely, consider the element of $\prod_{i\in \N}\Z/p^{m_i}$ given by the image of  
$$ (\ell_n\modulo \ell_i)_{i\in \N, i\neq n} \in\prod_{i\in \N,i\neq n}(\Z/\ell_i)^\times\twoheadrightarrow  \prod_{i\in \N,i\neq n}\Z/p^{m_i}\subset \prod_{i\in \N}\Z/p^{m_i}. $$
Evaluation at this element defines a ($p$-adic) measure which we will  denote by $\alpha_n\in \Z_p\llbracket \prod_{i\in \N} \Z/p^{m_i}\rrbracket  $. Since $\ell_n$ is orderly for $\rho$ we conclude in view of condition (6) and Lemma \ref{lem:fudgeinv}
 that 
 \begin{equation}\label{eq:element2}(-\alpha_n)^m P_{M,\ell_n}(\alpha_n^{-1})\in \Oo_{K_\mathfrak{p}}\llbracket \prod_{i\in \N} \Z/p^{m_i}\rrbracket \end{equation}
 is invertible (since each of its components are invertible). Thus by the horizontal norm relations (\ref{eq:hornormrelmodp}) satisfied by the theta elements  $\tilde{\theta}_{\ell_1\cdots \ell_n}^\pm$ and the general construction in \cite[Section 2.1.1]{KrizNordentoft23} (given by multiplying by the inverses of the projection of the  elements (\ref{eq:element2})) we obtain a compatible system of elements of $\mathcal{O}_{K_\mathfrak{p}} [ \prod_{1\leq i\leq n} \Z/p^{m_i}
    ]$ for each $n$ which defines a measure $\nu_{M,p}^\pm$. Observe that for $\chi \modulo q$ it holds that
    $$L^{(q)}(M\otimes\chi,s)=L(M\otimes\chi,s),$$
  since $\chi(n)=0$ for $(n,q)>1$. In view of Lemma \ref{lem:npBS} we conclude that the constructed measure satisfies the claimed interpolation property. 
\end{proof}
Note that we are suppressing the dependence on the sequence of orderly primes $\mathcal{L}=(\ell_n)_{n\in \N}$ and the component-wise projection 
$\rho: \prod_{n\in \N} (\Z/\ell_n)^\times\twoheadrightarrow \prod_{n\in \N}  \Z/p^{m_n}$. In some cases it is important to keep track of the choice of orderly primes in which case we will write 
$$\nu^\pm_{M,p,\mathcal{L}}\in \Oo_{K_\mathfrak{p}}\llbracket\prod_{n\in \N}\Z/p^{m_n} \rrbracket.$$
Finally, we will introduce some terminology: For a continuous character 
$$\chi:\prod_{n\in \N}\Z/p^{m_n} \rightarrow \overline{\Q}_p^\times,$$ 
we define the \emph{conductor} $\cond(\chi)$ as the conductor of the corresponding Dirichlet character $\chi\circ \rho$. Note that the conductor does not depend on the choice of $\rho$. 

\subsection{Example: holomorphic modular forms}  We will now explain that the conditions (1)-(7) are satisfied for a holomorphic newform $f$ of even weight $k\geq 2$ and level $N\geq 1$. Consider the Fourier expansion 
$$f(z)=\sum_{n\geq 1} a_f(n) q^n,\quad q=e(z)=e^{2\pi i z}.$$
Let $K_f$ denote the Hecke field of $f$ generated by the Fourier coefficients $a_f(n)$. We define the associated $L$-function  as the analytic continuation of
$$ L(f,s)= \sum_{n\geq 1} \frac{a_f(n)}{n^s},\quad \Re s>\frac{k+1}{2}.$$ 
The $L$-function satisfies a functional equation relating $L(f,s)$ and $L(\overline{f},k-s)$ and so we refer to $s=k/2$ as the \emph{central value}. We define the \emph{additive twist $L$-series of $f$} for $\frac{a}{q}\in \Q/\Z$ as the analytic continuation of 
$$ L(f,\tfrac{a}{q},s):= \sum_{n\geq 1} \frac{a_f(n)e(\frac{an}{q})}{n^s},\quad \Re s>\frac{k+1}{2},$$
which is exactly (\ref{eq:kloosttwist}) with $m=1$ but without the coprimality condition $(n,q)=1$. For a proof of analytic continuation of the additive twist $L$-series for any $\frac{a}{q}$ consult e.g.\ \cite[Section 3.3]{Nordentoft20.4}. 

Note that indeed this  $L$-function is associated with a motive $M_f/\Q$, see \cite{Scholl}. In order for $s=0$ to be the central value we define\footnote{Note that in standard notation for motives this corresponds to the shift $M_f(k/2)$.}
$$\lambda_{M_f}(n):=a_f(n) n^{-k/2},\quad n\geq 1.$$
In \cite[Section 3]{KrizNordentoft23} the conditions (1)-(6) were showed but for the twisted $L$-series \emph{with-out} the coprimality condition. This has the disadvantages of producing horizontal $p$-adic $L$-functions which are not compatible with direct sum of automorphic representations (or motives). We will thus need a slight modification of the argument in \cite{KrizNordentoft23} to show that the conditions (4) and (5) are satisfied for $M_f$ (with $m=1$). 
\begin{lemma}
The Dirichlet series $L(M_f, \tfrac{a}{q},1,s)$ admits analytic continuation to $\C$. There exists periods $\Omega_f^\pm\in \C^\times$ such that for any square-free $q\geq 2$ and $a\in (\Z/q)^\times$ it holds that 
       \begin{equation}\label{eq:algMf}
        \frac{1}{2\Omega_f^\pm }\left( L(M_f, \tfrac{a}{q}, 1, 0)\pm L(M_f, \tfrac{-a}{q}, 1, 0) \right)\in \mathcal{O}_{K_f}[\tfrac{1}{q}].
    \end{equation} 
\end{lemma}
\begin{proof}
To insert the coprimality condition we consider additive twists with larger denominators: for any divisor $d|q$ and $a\in (\Z/q)^\times$ we have by interchanging sums
\begin{align}
    \sum_{i=0}^{d-1} L(f,\tfrac{ad+iq}{qd},s)=  \sum_{n\geq 1}  \frac{a_f(n)}{n^s}  \left(\sum_{i=0}^{d-1} e(n(ad+iq)/qd) \right)&=
    \sum_{n\geq 1}  \frac{a_f(n)e(na/q)}{n^s}  \left(\sum_{i=0}^{d-1} e(ni/d) \right)\\
    &=d \sum_{n\geq 1,d|n}  \frac{a_f(n)e(na/q)}{n^s}.
    \end{align}
Thus it follows by inclusion-exclusion that
\begin{align}\label{eq:PIE}
    L(M_f,\tfrac{a}{q},1,s)= \sum_{d|q} \frac{\mu(d)}{d} \sum_{i=0}^{d-1} L(f,\tfrac{ad+iq}{qd},s+k/2), \quad \Re s>\frac{1}{2}.
\end{align}
This gives  the analytic continuation of the left-hand side. Finally, the algebraicity (\ref{eq:algMf}) follows from the expression (\ref{eq:PIE}) combined with classical results of Manin--Shimura, see \cite[Corollary 3.5]{KrizNordentoft23} for details.
\end{proof}

For a prime $p$ we define the \emph{mod $p$ representation associated with $f$} to be the residual representation \begin{equation}
    \label{eq:modp}\overline{\rho}_{f,\mathfrak{p}}:G_\Q\rightarrow \GL_2(\F_q),
\end{equation} as in Example \ref{ex:newform} where $\mathfrak{p}=K_f\cap \overline{\mathfrak{p}}$ with $\overline{\mathfrak{p}}\subset \overline{\Z}$ the prime ideal above $p$ determined by the choice of embedding (\ref{eq:embed}). Then condition (6) is satisfied with $$\rho=\overline{\rho}_{f,\mathfrak{p}}\otimes \chi_\mathrm{cyc}^{-k/2},$$ 
where $\chi_\mathrm{cyc}$ denotes the $p$-cyclotomic character.  Note here that the set of orderly primes of order $m$ is unchanged by twists by a power of the cyclotomic character since $\chi_\mathrm{cyc}(\Frob_\ell)=1$ for $\ell\equiv 1\modulo p$. This means that if the set of orderly primes (of order $1$) for $\overline{\rho}_{f,\mathfrak{p}}$ has positive density then condition (7) is also satisfied. In this case we conclude from Theorem \ref{thm:padicL} the existence of a horizontal $p$-adic $L$-function $\nu^\pm_{f,p,\mathcal{L}}$ associated to $f$, a choice of sign $\pm$ and a sequence of orderly primes $\mathcal{L}$ satisfying the interpolation formula (\ref{eq:interpolation}). 

\subsubsection{Higher rank}\label{sec:higherrank} If $f_1,\ldots, f_n$ are even weight newforms we want to construct the horizontal $p$-adic $L$-function associated with the higher rank automorphic representation $\pi= \pi_{f_1}\boxplus \cdots \boxplus \pi_{f_n}$ f $\GL_{2n}(\Q)$, where $\pi_{f_i}$ denotes the automorphic representation of $\GL_{2}(\Q)$ associated with $f_i$. In this case the automorphic $L$-function is
\begin{equation}\label{eq:MpiLfnc} L(\pi,s)=L(f_1,s+\tfrac{k_1-1}{2})\cdots L(f_n,s+\tfrac{k_n-1}{2}), \end{equation}
and so we define the direct sum  $M_\pi:=M_{f_1}\oplus \cdots \oplus M_{f_n}$ by the convolution product 
$$ \lambda_{M_\pi}(n):= \left((\,\cdot\,)^{-k_1/2}a_{f_1}(\,\cdot\,)\right)\ast \cdots \ast \left((\,\cdot\,)^{-k_n/2} a_{f_n}(\,\cdot\,)\right),$$
so that indeed
\begin{equation}
    \label{eq:L1/2}
L(M_\pi,s)=L(\pi,s-\tfrac{1}{2}). \end{equation}
This clearly satisfies the conditions (1)-(3). One checks directly, using crucially the coprimality condition, that  
\begin{align}
L(M_\pi,\tfrac{a}{q}, n, s)=\sum_{\substack{a_1,\ldots,a_n\in (\Z/q)^\times,\\ a_1\cdots a_n\equiv a\modulo q }} \prod_{i=1}^n L(M_{f_i},\tfrac{a_i}{q}, 1, s),\quad \Re s>\frac{k+1}{2},
\end{align}
so that condition (4) is satisfied with $m=n$. Furthermore, this implies
\begin{equation*}
\begin{split}
& \sum_{\substack{
  a_1 , \ldots , a_n \in(\mathbb{Z} / q)^\times \\
  a_1 \cdots a_n \equiv a \bmod{q}}} \quad \prod_{i=1}^n \left( L(M_{f_i}, \tfrac{a_i}{q},1,s) \pm L(M_{f_i},\tfrac{-a_i}{q},1,s) \right)\\
  =\quad & \sum_{\substack{
     a_1 , \ldots , a_n \in(\mathbb{Z} / q)^\times \\
  a_1 \cdots a_n \equiv a \bmod{q}
  }} \quad \sum_{e_1 ,\ldots , e_n  \in \left\{ 0,1 \right\}} \quad  \prod_{i=1}^n(\pm 1)^{e_i} L(M_{f_i},\tfrac{(-1)^{e_i} a_i}{q},1,s)\\
  =\quad & 2^{n -1}\left[\sum_{\substack{
a_1 ,\ldots , a_n \\ a_1 \cdots a_n \equiv a \bmod{q} 
}} \quad \prod_{i=1}^n L(M_{f_i}, \tfrac{a_i}{q},1,s)\quad  \pm \sum_{\substack{
a_1 ,\ldots , a_n \\ a_1 \cdots a_n \equiv -a \bmod{q} 
  }} \quad \prod_{i=1}^n L(M_{f_i}, \tfrac{a_i}{q},1,s)\right]\\
  = \quad &2^{n-1} \left( L(M_{\pi}, \tfrac{a}{q},n,s) \pm L(M_{\pi},\tfrac{-a}{q},n,s) \right),
\end{split}
\end{equation*}
so that condition (5) is  satisfied with $\Omega^\pm =\prod_{i=1}^n \Omega_{f_i}^\pm $. 

Let  $\rho_1,\ldots, \rho_n$ be the  mod $p$ representations associated with $f_1,\ldots , f_n$. Then condition (6) is satisfied  with the Galois  representation $$\rho=(\rho_1\otimes \chi_\mathrm{cyc}^{-k_1/2})\oplus \cdots\oplus (\rho_n\otimes \chi_\mathrm{cyc}^{-k_n/2}),$$ which follows directly from the equalities (\ref{eq:MpiLfnc}) and (\ref{eq:L1/2}). Assuming now that the set of orderly primes for $\rho$, or  equivalently for $\rho_1\oplus \cdots \oplus \rho_n$, has positive density so that condition (7) is also satisfied, we let $\mathcal{L}$ denote the sequence of such joint orderly primes (cf.\ \cite[Definition 4.12]{KrizNordentoft23}). Then we obtain, for each choice of sign $\pm$, the horizontal $p$-adic $L$-function  of $M_\pi$  by Theorem \ref{thm:padicL} which in this case is exactly equal to the product: 
\begin{equation}
    \nu_{M_\pi,p,\mathcal{L}}^\pm=\prod_{i=1}^n \nu^\pm_{f_i,p,\mathcal{L}}.
\end{equation}
\begin{remark}
    We expect that the above formalism is also relevant for \emph{cuspidal} automorphic representations of higher rank groups satisfying appropriate condition, in particular  having the central value as a critical value. For instance, Dimitrov--Januszweski--Raghuram defined  (vertical) $p$-adic $L$-functions for  $\GL_{2n}$-automorphic representations admitting a Shalika model \cite[Theorem B]{DiJaRa20} with an interpolation formula compatible with Theorem \ref{thm:padicL} in the case $m=n$.
\end{remark}
\section{Applications to non-vanishing}\label{sec:prop}
In this section we will apply the results from the previous sections to obtain a number of non-vanishing results. We will obtain results both for the \emph{arithmetic} family of Galois characters of $G$-extension with $G$ a fixed finite abelian group and for the \emph{analytic} family of degree $d$ characters with $d\geq 2$ fixed. 

\subsection{Simultaneous non-vanishing in abelian extensions}
For $G$  a finite abelian group, we define the following family of fields for $X\geq 1$:
\begin{equation}
    \mathcal{F}_{G}(X):=\{  F/\Q \text{ \rm Galois}\mid \Gal(F/\Q)\cong G, \disc(F)\leq X \}.
\end{equation}
Recall that we defined 
\begin{equation}
    a(G):=\frac{p}{|G|(p-1)},\quad b(G):= \frac{p^m-1}{p-1} 
\end{equation}
where $p$ is the smallest prime number dividing $|G|$, and $p^m$ is the largest $p$-power order cyclic factor of $G$. Then it is a result of Wright \cite{Wright} that there exists a constant $c(G)>0$ such that
$$|\mathcal{F}_G(X)|=(c(G)+o(1))X^{a(G)}(\log X)^{b(G)-1},\quad \text{as }X\rightarrow \infty. $$
Given a Galois representation $$\rho:G_\Q\rightarrow \GL_n(\F),$$
with $\F$ a field (e.g.\ a finite field or $\C$), we denote by 
$$F_\rho:= \overline{\Q}^{\ker \rho},$$  
the field cut out by $\rho$.  Recall the basic fact that the extension $F_\rho/\Q$ is Galois and there is a natural identification $\Gal(F_\rho/\Q)\cong \im(\rho)$. 
\begin{lemma}\label{lem:taub}
    Let $\mathcal{P}$ be a subset of primes congruent to $1$ modulo $p$ with natural density $\alpha>0$ among all primes congruent to $1$ modulo $p$. Assume furthermore that there exists 
    $\ell\in \mathcal{P}$ such that $\ell\equiv 1\modulo p^m$. Then it holds that 
    $$|\{ F_\chi\in \mathcal{F}_{\Z/p^m}(X) \mid   \cond(\chi)\text{ \rm divides }\prod_{\ell\in \mathcal{P} }\ell\}|\gg \frac{X^{a(\Z/p^m)}}{(\log X)^{\alpha\cdot  b(\Z/p^m)-1 }}, $$
    as $X\rightarrow \infty$. 
\end{lemma}

Note that if $F_{\chi}=F_{\chi'}$, then, since $\mathbb{Z}/p^m$ is cyclic, $\chi$ and $\chi'$ are Galois conjugates so they have the same conductors. Hence the condition $\mathrm{cond}(\chi)\mid \prod_{\ell\in\mathcal{P}} \ell$ is well-defined, i.e. it does not matter which characters we use to define the fields in $\mathcal{F}_{\mathbb{Z}/p^m}(X)$.

\begin{proof}
    Let $\ell_0\in \mathcal{P}$ be such that $\ell_0\equiv 1\modulo p^m$ and let $\chi_0\modulo \ell_0$ be of order $p^m$. Then for any order $p$ Dirichlet character $\chi$ such that $\cond(\chi)$ divides $\prod_{\ell\in \mathcal{P},\ell\neq \ell_0} \ell$, then $\chi\chi_0$ has order $p^m$ and by the conductor-discriminant formula \cite[Ch. VII (11.9)]{neukirch} 
    $$\disc(F_{\chi\chi_0})= \prod_{j=1}^{p^m}\cond((\chi\chi_0)^j)\ll_{\ell_0} \cond(\chi)^{p^m-p^{m-1}}, $$
    using that 
    $$\cond((\chi\chi_0)^j)\ll_{\ell_0} 1,$$
    for $p|j$. Now the result follows from a standard  Tauberian argument, see \cite[Lemma 5.7]{KrizNordentoft23}.  
\end{proof}

We now prove a lemma that establishes sufficient conditions for the existence of orderly primes for twists of mod $p$ representations by Dirichlet characters. This result is important when extending our arguments from cyclic groups to general abelian groups. We first set up some notation. 

Recall from \eqref{eq:embed} that we have fixed an embedding $\overline{\mathbb{Q}}\hookrightarrow\overline{\mathbb{Q}}_p$, and that this embedding determines a prime ideal $\overline{\mathfrak{p}}$ of $\overline{\mathbb{Z}}$ lying over $p$. If $F/\mathbb{Q}$ is a finite Galois extension, and $\chi: \mathrm{Gal}(F/\mathbb{Q})\rightarrow \mathbb{C}^{\times}$ is a character, then $\chi$ takes values in $\overline{\mathbb{Z}}$, so by reducing modulo $\overline{\mathfrak{p}}$, we obtain a character $\overline{\chi}: \mathrm{Gal}(F/\mathbb{Q})\rightarrow \overline{\mathbb{F}}_q^{\times}$. Hence if $\rho: G_{\mathbb{Q}}\rightarrow \mathrm{GL}_n(\mathbb{F}_q)$ is a mod $p$ representation, we can form the tensor product $\rho\otimes\overline{\chi}$, and the following lemma gives sufficient conditions for the existence of infinitely many orderly primes for this representation. 

\begin{lemma}\label{lem:Galoislemma}
    Let $\rho:G_\Q\rightarrow \GL_n(\F_q)$ be a mod $p$ Galois representation. Assume that the set of orderly primes of order $m$ for $\rho$ has positive density among all primes. Let $F/\Q$ be an abelian extension such that $(\disc(F),p\cdot \disc(F_\rho))=1$.  Then the set of orderly primes of order $m$ for 
 \begin{equation}\label{eq:combinedrep2}
  \bigoplus_{\chi\in \widehat{\Gal(F/\Q)}}\rho\otimes \overline{\chi}\end{equation} also has positive density.  
\end{lemma}

\begin{proof}
   To prove the claim, consider the $[F:\Q]$-dimensional Galois representation:
$$\overline{\rho}_{F}:= \bigoplus_{\chi\in \widehat{\Gal(F/\Q)}} \overline{\chi}, $$
and let $\chi_{\mathrm{cyc},p^m}:G_\Q\rightarrow (\Z/p^m)^\times$ be the cyclotomic character modulo $p^m$. Note that the field cut out by $\overline{\chi}$ is a subfield of $F$ and so, 
by the coprimality assumption, the ramification of $\overline{\rho}_{F} $ and that of $\rho\oplus \chi_{\mathrm{cyc},p^m}$ are disjoint. Since the class number of $\Q$ is $1$ it follows from basic Galois theory \cite[Lemma 4.14]{KrizNordentoft23} that
$$ \im (\rho\oplus \chi_{\mathrm{cyc},p^m}\oplus   \overline{\rho}_{F})= \im (\rho\oplus \chi_{\mathrm{cyc},p^m})\times \im(\overline{\rho}_{F}). $$
We know by assumption that $\im (\rho\oplus \chi_{\mathrm{cyc},p^m})$ contains an element of the shape 
$$(g_0,1) \in \GL_{n}(\F_q)\times (\Z/p^m)^\times,$$ where $g_0$  does not have $1$ as an eigenvalue. We conclude by the Chebotarev density theorem that for a positive proportion of unramified primes $\ell$ the image of $\Frob_\ell$ under $\rho\oplus \chi_{\mathrm{cyc},p^m}\oplus  \overline{\rho}_{F}$  is contained in the conjugacy class of  
\begin{equation} 
\label{eq:g0}(g_0,1,1,\ldots ,1)\in \GL_{n}(\F_q)\times (\Z/p^m)^\times \times \prod_{\chi\in \widehat{\Gal(F/\Q)}} \overline{\mathbb{F}}_p^{\times}.\end{equation}
 For any such prime $\ell$ and $\chi\in\widehat{\Gal(F/\Q)} $ it holds that $\ell\equiv 1\modulo p^m$ and $\overline{\chi}(\Frob_\ell)=1$. This latter property implies that the conjugacy class of   $(\rho\otimes \overline{\chi})(\Frob_\ell)$ equals that of $\rho(\Frob_\ell)$ which, in turn, equals that of $g_0$. By assumption $g_0$ does not have $1$ as an eigenvalue. This shows that the positive proportion of primes corresponding to the conjugacy class (\ref{eq:g0}) all are orderly primes of order $m$ for the mod $p$ Galois representation (\ref{eq:combinedrep2})  and so we conclude.
\end{proof}
We obtain the following new propagation of non-vanishing result.
\begin{theorem} \label{thm:propagG}
Let $f_1 ,\dots, f_n$ be newforms of even weights $k_i$ respectively. Assume that
\begin{equation}
    \label{eq:assump1}L(f_1,k_1/2)\cdots L(f_n,k_n/2)\neq 0.
\end{equation}
Let $G$ be a non-trivial finite abelian group. Assume that for any factor $\Z/p^m$ in the primary decomposition of $G$ the following holds: the set of  orderly primes of order $m$ for $\rho_1\oplus \cdots\oplus \rho_n$ has positive density among all primes, where $\rho_i$ denotes the mod $p$ representations associated with $f_i$ as in (\ref{eq:modp}). Then there exists a constant $\kappa<1$ such that
\begin{align}
|\{ F\in \mathcal{F}_G(X)\mid L(f_1,\chi,k_1/2)\cdots L(f_n,\chi,k_n/2) \neq 0\text{ \rm for all $\chi\in\widehat{\Gal(F/\Q)}$} \}|\gg
\frac{X^{a(G)}}{(\log X)^{\kappa}},\end{align} 
as $X\rightarrow \infty$. Here $\kappa$ and the implied constant might depend on $G$ and $f_1,\ldots, f_n$.
\end{theorem}

\begin{proof}
 Let 
 $$G\cong \Z/q_1\times \cdots\times  \Z/q_s,$$
 be the primary decomposition of $G$ so that  $q_i$ are powers of, not necessarily distinct, primes. We proceed by induction on the number $s$ of cyclic factors. It will be convenient to take a slightly stronger induction assumption, namely that  the number fields $F\in \mathcal{F}_G(X)$ all satisfy that $\disc(F)$ is coprime to some fixed integer  $S\geq 1$. It is also convenient to let $G=1$ be the base case in which it follows from the non-vanishing assumption. 
 For the induction step, let 
 $$G_0:=\Z/q_1\times \cdots \times \Z/q_{s-1} \leq G,$$
 and assume that $q_s=p^m$ with $p$ the smallest prime dividing $|G|$.
 By induction we can, in particular, find an abelian $G_0$-extension $F_0/\Q$ with discriminant coprime to $S=|G|\cdot N_1\cdots N_n$ where $N_i$ denotes the level (or conductor) of $f_i$. Let $\rho_1,\ldots, \rho_n$ be the residual representations associated with $f_1,\ldots, f_n$ as in the assumption of the theorem, so that the set of orderly primes of order $m$ for $$\rho=\rho_1\oplus \cdots \oplus \rho_n:G_\Q\rightarrow \GL_{2n}(\overline{\F}_p),$$
 has positive density. By Lemma \ref{lem:Galoislemma} we conclude that the set of orderly primes of order $m$ for 
 \begin{equation}\label{eq:combinedrep}
  \bigoplus_{\chi'\in \widehat{\Gal(F_0/\Q)}}\rho\otimes \overline{\chi}'\cong \bigoplus_{\substack{\chi'\in \widehat{\Gal(F_0/\Q)}\\1\leq i\leq n}}\rho_i\otimes \overline{\chi}'\end{equation} has positive density among all primes. Observe that the mod $p$ representation $\rho_i\otimes \overline{\chi}'$ is exactly the mod $p$ representation for the twisted newform $f_i\otimes \chi'$ both defined via the prime above $p$ determined by (\ref{eq:embed}). We remark that $f_i\otimes\chi'$ is again a newform because the conductor of $\chi'$ divides $\mathrm{disc}(F_0)$ by the conductor-discriminant formula so $\mathrm{cond}(\chi')$ is coprime to the level of $f_i$.

 Let $\ell_1<\ell_2<\ldots$ be the orderly primes of order $1$ for the Galois representation (\ref{eq:combinedrep}) which are coprime to $|G|\cdot  \disc(F_0)$ and consider for each $f_i$ and $\chi'\in \widehat{\Gal(F_0/\Q)}$ the associated horizontal $p$-adic $L$-function as in Theorem \ref{thm:padicL}:
$$ \nu_{f_i\otimes \chi',p}\in R\llbracket \prod_{i\in\N} \Z/p^{m_i}\rrbracket, $$
where $m_i=v_p(\ell_i-1)\geq 1$ (note that by assumption  $m_i\geq m$ for infinitely many $i\in \N$) and $R=( \Oo_{(K_1)_{\mathfrak{p}_1}},\ldots ,\Oo_{(K_n)_{\mathfrak{p}_n}})\subset \Oo_{\C_p}$ where $K_i$ denotes the Hecke field of $f_i$ and $\mathfrak{p}_i=K_i\cap \overline{\mathfrak{p}}$ with $\overline{\mathfrak{p}}\subset \overline{\Z}$ determined by (\ref{eq:embed}). Consider the product
$$ \nu:= \prod_{1\leq i\leq n}\,\,\prod_{\chi'\in \widehat{\Gal(F_0/\Q)}} \nu_{f_i\otimes \chi',p}\in R\llbracket \prod_{i\in\N} \Z/p^{m_i}\rrbracket,$$
and the pushforward
$$\pi_\ast(\nu)\in R\llbracket \prod_{i\in \N} \Z/p^{\min(m,m_i)}\rrbracket  $$
along the natural projection map
 $$ \pi: \prod_{i\in \N} \Z/p^{m_i}\twoheadrightarrow \prod_{i\in \N} \Z/p^{\min(m,m_i)}.$$
By the interpolation property (\ref{eq:interpolation}) and the assumption (\ref{eq:assump1}) it holds that $\pi_\ast(\nu)(\1)=\nu(\1)\neq 0$ so that we can apply  Theorem \ref{thm:structurenew}. Let  $M_0:=M_{\pi_\ast(\nu)}$ be the finite subgroup as in the statement of the theorem which we will considered to be fixed. Let $\chi$ be a character of $\prod_{i\in \N} \Z/p^{\min(m,m_i)}$ of order $p^m$ with conductor coprime to the conductor of any character in $M_0$. By Theorem \ref{thm:structurenew} we conclude that there exists $\chi_0\in M_0$ so that  
 $$\pi_\ast(\nu)((\chi\chi_0)^j)=\nu((\chi\chi_0)^j)\neq 0,\quad 1\leq j\leq p^m.$$ 
Note that the conductor of $\chi\chi_0$ is coprime to $\disc(F_0)$. Recalling the definition of the measures this implies  
 $$\nu_{f_i\otimes \chi',p}((\chi\chi_0)^j)\neq 0,\quad \text{for all}\quad\chi'\in \widehat{\Gal(F_0/\Q)},\, i=1,\ldots,n,\, j=1,\ldots,p^m.  $$
 By the interpolation property of the horizontal $p$-adic $L$-function, this is equivalent to
 $$ L(f_i, \chi'(\chi\chi_0)^j,k_i/2)\neq 0.$$
 As above, let $F_{\chi\chi_0}$ denote the cyclic order $p^m$ number field cut out by $\chi\chi_0$. Then by the coprimality assumption, we conclude from basic Galois theory that the compositum $F=F_0F_{\chi\chi_0}$ has Galois group exactly $G_0\times \Z/p^m\cong G$. Note that  $\chi\chi_0$ uniquely determines $\chi$, and, by the conductor discriminant formula, 
$$\disc(F)= \prod_{\chi'\in \widehat{\Gal(F_0/\Q)}}\,\,\prod_{j=1}^{p^m} \cond(\chi' (\chi\chi_0)^j)\ll_{F_0,G} \left(\prod_{j=1}^{p^m}\cond( (\chi\chi_0)^j)\right)^{|G_0|}= \disc(F_{\chi\chi_0})^{|G_0|},$$
and since $M_0$ is considered fixed we have $\disc(F_{\chi\chi_0})\asymp \disc(F_{\chi})$. The result now follows from Lemma \ref{lem:taub} since indeed 
$$|G_0|(p^m-p^{m-1})=|G|(p-1)/p=1/a(G),$$
and the set of orderly primes of order $m$ for $\rho$ has positive density.
\end{proof}
Combining with the results from Section \ref{sec:orderly} we obtain the following result.
\begin{theorem}
\label{thm:propagGappl}
Let $f_1 ,\dots, f_n$ be newforms of even weights $k_i$ respectively. Assume that
$$L(f_1,k_1/2)\cdots L(f_n,k_n/2)\neq 0.$$
Then there exists a $M\geq 1$ such that if $G$ is a finite abelian group such that any prime divisor of $|G|$ is at least $M$ then the following holds: there exists a constant $\kappa<1$ such that

\begin{align}\label{eq:FGfi}
|\{ F\in \mathcal{F}_G(X)\mid L(f_1,\chi,k_1/2)\cdots L(f_n,\chi,k_n/2) \neq 0\text{ \rm for all $\chi\in\widehat{\Gal(F/\Q)}$} \}|\gg
\frac{X^{a(G)}}{(\log X)^{\kappa}},\end{align} 
as $X\rightarrow \infty$.

Furthermore, if $f_i=f_{E_i}$ is associated with an elliptic curve for all $i=1,\ldots,n$ then one can take $M=\max\{167, 3n+2\}$, and if $E_1,\ldots,E_n$ are all semi-stable, one can take $M= \max\{11,n+1\}$. If $n=2$ then it suffices that the mod $p$ representations associated with $f_1,f_2$ are absolutely irreducible for all prime divisor $p$ of $|G|$.
\end{theorem}
\begin{proof}
   From the propagation of non-vanishing from Theorem \ref{thm:propagG} it suffices that for all cyclic factors $\Z/p^m$ in $G$ there is a positive proportion of orderly primes of order $m$ for the mod $p$ representation $\rho_1\oplus \cdots \oplus \rho_n$ (associated to $f_1\boxplus \cdots \boxplus f_n$). In view of Theorem \ref{thm:non-van} this is true for $p$ sufficiently large. In the case of elliptic curves it suffices by Theorem  \ref{thm:ordrlyellipticcurves} that $p\geq \max\{167, 3n+2\}$. Finally, if $n=2$ then absolute irreducibility of $\rho_1$ and $\rho_2$ suffices by Proposition \ref{prop:two-non-van}.
\end{proof}
\begin{remark}
    Note that for $F/\Q$ abelian the condition:
    $$
``L(f_1,\chi,k_1/2)\cdots L(f_n,\chi,k_n/2) \neq 0\text{ \rm for all $\chi\in\widehat{\Gal(F/\Q)}$''} $$ means exactly that $$L(f_1/F,k_1/2)\cdots L(f_n/F,k_n/2) \neq 0,$$
where $L(f_i/F,s)$ denotes the base change $L$-function of $f_i$ to $F$.
\end{remark}

\subsection{Non-vanishing for characters of fixed order} \label{sec:nonvanishing}
We now present applications of the results in Section \ref{sec:orderly} to simultaneous non-vanishing of twisted central $L$-values by characters of \emph{fixed order} for holomorphic newforms of even weight.  Recall here the definition of the set $\mathcal{K}_d(X)$ of characters of order $d$ and conductor $\leq X$ as defined in (\ref{eq:Kd}).

In \cite[Section 5.2]{KrizNordentoft23} a general ``propagation of vanishing''-principle was proved in the presence of a horizontal $p$-adic $L$-function. As illustrated in the previous section the existence follows if there is a positive proportion of orderly primes for the corresponding mod $p$ representation(s).  Using the existence of the horizontal $p$-adic $L$-function, the following result was obtained in \cite{KrizNordentoft23}.

\begin{theorem}[\!\text{\!\cite[Corollary 5.19]{KrizNordentoft23}}] \label{thm:propag}
Let $f_1 ,\dots, f_n$ be newforms of even weights $k_i$ respectively. Assume that $$L(f_1,k_1/2)\cdots L(f_n,k_n/2)\neq 0.$$
Let $d\geq 2$ be an integer and let $p$ be a prime dividing $d$ with $p^m|\!|d$. Assume that the set of  orderly primes of order $m$ for $\rho_1\oplus \cdots\oplus \rho_n$ has positive density among all primes, where $\rho_i$ denotes the mod $p$ representation associated with $f_i$ as in (\ref{eq:modp}). Then there exists a constant $\kappa<1$ such that
\begin{align}
|\{ \chi\in \mathcal{K}_d(X)\mid  L(f_1,\chi,k_1/2)\cdots L(f_n,\chi,k_n/2)\neq 0 \}|\gg
\frac{X}{(\log X)^{\kappa}},\quad \text{as }X\rightarrow \infty.\end{align} 
Here $\kappa$ and the implied constant might depend on $d$ and $f_1,\ldots, f_n$.
\end{theorem}
In view of the results in Section \ref{sec:orderly} we obtain a proof of Theorem \ref{thm:nonvanishing}  from the introduction.
\begin{proof}[Proof of Theorem \ref{thm:nonvanishing}]
By Theorem \ref{thm:propag} the lower bound (\ref{eq:conclusionthm}) holds as long as there is a prime $p$ with $p^m|\!| d$ such that the set of orderly primes of order $m$ for $\rho_1\oplus \cdots \oplus \rho_n$ has positive density, where $\rho_1,\ldots, \rho_n$ denote mod $p$ representations associated with $f_1,\ldots, f_n$. In view of Theorem \ref{thm:non-van} this is true for $p$ sufficiently large. This proves the first part of the theorem. Furthermore, for $n=2$ we conclude from Proposition \ref{prop:two-non-van} that when $\rho_1,\rho_2$ are both absolutely irreducible then the set of orderly primes of order $m$ for $\rho_1\oplus \rho_2$ has positive density. This yields the last part of the theorem in view of Theorem \ref{thm:propag}.    \end{proof}

\section{Arithmetic applications}
In this section we will present arithmetic applications of the non-vanishing results for central $L$-values. For general newforms these applications will be to finiteness of so-called \emph{Selmer groups}, which are subgroups in Galois cohomology  cut out by certain local conditions. When the cusp form is associated with an elliptic curve, and more generally modular abelian varieties, this serves as a (more easily) computable proxy for the rational points and will yield results on the rational points themselves. 

\subsection{Kato's divisibility theorem and its consequences} Let $f$ be a holomorphic newform of even weight $k$ and Hecke field $K$. Let $p$ be a prime and put $\mathfrak{p}=\overline{\mathfrak{p}}\cap K$ with $\overline{\mathfrak{p}}\subset\overline{\Z}$ determined by the embedding (\ref{eq:embed}). Let    $T_f\subset V_{f,\mathfrak{p}}(k/2)\cong K_\mathfrak{p}^2$ be a  $\mathcal{O}_{K_\mathfrak{p}}$-lattice of rank 2 which is preserved by the action of the Galois group $G_\Q$ under the image of the renormalized associated $p$-adic Galois representation
$$ \rho_{f,\mathfrak{p}}\otimes \chi_\mathrm{cyc}^{-k/2}:G_\Q\rightarrow \GL_2(\mathcal{O}_{K_\mathfrak{p}})\subset \GL_2(K_\mathfrak{p}),$$
where $\chi_\mathrm{cyc}:G_\Q\rightarrow \Z_p^\times \subset \mathcal{O}_{K_\mathfrak{p}}^\times$ denotes the cyclotomic character. For an abelian extension $F/\Q$ we consider the associated \emph{$p$-adic Selmer group} inside the Galois cohomology of $F$ acting on the quotient $V_{f,\mathfrak{p}}(k/2)/T_f$: 
\begin{equation}S(T_f/F)\subset H^1(F,V_{f,\mathfrak{p}}(k/2)/T_f).\end{equation} 
This is a $\mathcal{O}_{K_\mathfrak{p}}$-module with a natural $\Gal(F/\Q)$-action; we will refer to \cite[Section 14.1]{Kato} for details on the construction. Given a Galois characters $\chi:\Gal(F/\Q)\rightarrow \overline{\Q}^\times $ one defines the $\chi$-isotypic component $S(T_f/F)^{(\chi)}$ via the $\Gal(F/\Q)$-action; again we refer to \cite[Section 14.1]{Kato} for details. 
It is a fundamental result of Kato that the non-vanishing of the  twisted central $L$-value controls the size of Selmer groups, confirming one direction of the equivariant Bloch--Kato conjecture in this case:
\begin{theorem}[\!{\!\cite[Theorem 14.2(2)]{Kato}}]\label{thm:Kato}
  If $L(f,\chi,k/2)\neq 0$, then $S(T_f/F)^{(\chi)}$ is finite.  
\end{theorem}
Combining this with our Theorem \ref{thm:nonvanishing} we obtain the following arithmetic consequence. Recall the definition of the set of order $d$ characters  $\mathcal{K}_d(X)$ and the field $F_\chi$ cut out by a Galois (or Dirichlet) character $\chi$.
\begin{corollary}
    Let $f_1,\ldots, f_n$ be holomorphic newforms of even weights $k_i$. Assume that 
    $$L(f_1,k_1/2)\cdots L(f_n,k_n/2)\neq 0.$$  
    Let $p$ be a prime and for $F/\Q$ abelian let $S(T_i/F)$ denote the $p$-adic Selmer group associated to $f_i$ as above. Then there exists a subset $\mathcal{D}\subset \N$ of natural density $1$ so that for any $d\in \mathcal{D}$ the following holds: there exists a constant $\kappa<1$ so that 
    \begin{align}\label{eq:conclusioncor}
|\{ \chi\in \mathcal{K}_d(X) \mid S(T_i/F_\chi)^{(\chi)}\text{ \rm is finite for all } i=1,\ldots, n \}| \gg \frac{X}{(\log X)^{\kappa}},
\end{align}  
as $X\rightarrow \infty$.
    \end{corollary}
    \begin{proof}
 This follows directly by combining Theorem \ref{thm:Kato} and Theorem \ref{thm:nonvanishing} observing that for $M\geq 1$ fixed the set of integers with a prime divisor larger than $M$ has natural density one among all positive integers.       
    \end{proof}
\subsubsection{Rational points on abelian varieties} 
Now we will restrict to the case that $f$ is a newform of weight $2$ with Hecke field $K_f$. For each $\sigma\in\Gal(K_f/\Q)$ one obtains a newform $f^\sigma$ by acting on the Fourier coefficients and it is a theorem of Shimura \cite{Shimura77} that 
$$L(f,1)\neq0 \Leftrightarrow L(f^\sigma,1)\neq 0.$$ 
There is an associated simple abelian variety $A_f/\Q$ of dimension $ [K_f:\Q]$ and the Hasse--Weil $L$-function $L(A_f,s)$ admits analytic continuation whic is equal, up to finite Euler product non-vanishing at $s=1$, to the product
\begin{equation}
    \label{eq:HasseWeilfactor}\prod_{\sigma\in \Gal(K_f/\Q)} L(f^\sigma, s),
\end{equation}
where $f^\sigma$ denotes a Galois conjugate of $f$ (see \cite{Ribet04} and the references therein).  
Recall that  an abelian variety $A/\Q$ is of \emph{$\GL_2$-type} if there  exists a $\Q$-algebra embedding $K\subset \mathrm{End}_\Q(A)\otimes_\Z \Q$ where $K$ is a field of degree equal to the dimension of $A$. It follows by combining \cite[Theorem 4.4]{Ribet04} with \cite{KhareWintenberger2009a,KhareWintenberger2009b} that any simple abelian variety $A/\Q$ of $\GL_2$-type is isogenous to $A_f$ for some newform $f$ of weight 2. In particular, an abelian variety $A/\Q$ is isogenous to a product of $\GL_2$-type abelian varieties if and only there is a surjective morphism $J_1(N)\twoheadrightarrow A$ from the Jacobian of a modular curve $X_1(N)$, which  we referred to in the introduction as a \emph{modular abelian variety}. 
When $A/\Q$ is simple of $\GL_2$-type,  the endomorphism algebra $K=\mathrm{End}_\Q(A)\otimes_\Z \Q$ is itself a field of degree equal to the dimension of $A$  (see \cite[Theorem 2.1]{Ribet04}) which we will refer to as the \emph{coefficient field of $A$}. This is isogeny invariant and up to isogeny we can assume that $\mathrm{End}_\Q(A)=\mathcal{O}_{K}$ is the maximal order in $K$.  
In this case $A$ is a module over $\mathcal{O}_{K}$ and so for any prime $\mathfrak{p}$ in $K$ we get a $G_\Q$-module $A[\mathfrak{p}]$ by taking the kernel of $\mathfrak{p}$. This is a $2$-dimensional Galois representation over the residue field $\F_\mathfrak{p}$, see e.g.\ the discussion in \cite[Section 3]{Ribet04} for more details.   

Now let $A_f/\Q$ be the simple abelian variety associated with a weight 2 newform $f$. For $F/\Q$ an abelian extension it follows from the Mordel--Weil Theorem that the $F$-rational points $A_f(F)$ is a finitely generated abelian group admitting an action of $\Gal(F/\Q)$. One  obtains a finite dimensional complex representation $A_f(F)\otimes_\Z \C$ (with the trivial action on the second component) with an isotypic  decomposition 
\begin{equation} \label{eq:decomp}A_f(F)\otimes_\Z \C = \bigoplus_{\chi\in \widehat{\Gal(F/\Q)}} \left(A_f(F)\otimes_\Z \C\right)^{(\chi)}, \end{equation}
in terms of complex characters $\chi:\Gal(F/\Q)\rightarrow \C^\times$.
As spelled out in \cite[Section 3]{MV02} (see the equation before (3.2) as well as the discussion in \cite[Section 1.5]{Merel01}) it follows from a standard Kummer exact sequence that we have the following consequence of Kato's work \cite{Kato}: 
\begin{equation}
    \label{eq:Katocons}\prod_{\sigma\in \Gal(K/\Q)}L(f^\sigma, \chi,1)\neq 0 \text{ implies that }\dim\left(A_f(F)\otimes_\Z \C\right)^{(\chi)}=0.
\end{equation} 
We summarize the above discussion as follows.
\begin{corollary}[\!{Cf. \cite[Corollary 14.3]{Kato},\cite[equation (3.2)]{MV02}}]\label{cor:Kato}
  Let $f$ be a newform of weight $2$ with associated abelian variety $A_f/\Q$ and coefficient field $K$. Let $F/\Q$ be an abelian extension. Assume that $L(f^\sigma,\chi,1)\neq0$ for all $\sigma\in \Gal(K/\Q)$ and all $\chi\in \widehat{\Gal(F/\Q)}$. Then the group of $F$-rational points $A_f(F)$ is finite. 
\end{corollary}
\begin{proof}
    This follows directly by combining the implication (\ref{eq:Katocons})  and the decomposition (\ref{eq:decomp}).
\end{proof}

Combining the above with the new propagation of non-vanishing result in Theorem \ref{thm:propagG} and our results on the existence of orderly primes in Theorem \ref{thm:non-van}, we obtain proof of the results on Diophantine rank stability claimed in the introduction.
\begin{proof}[Proof of Theorem \ref{thm:abelianvar}]
 Since $A/\Q$ is  isogenous to a product of $\GL_2$-type abelian varieties, we can by the above discussion assume that $A=A_{f_1}\times\cdots\times A_{f_n}$ for newforms $f_i$ of weight $2$ with coefficient fields $K_i$. By the assumption of analytic rank $0$ and the expression (\ref{eq:HasseWeilfactor}) for the Hasse--Weil $L$-function  we can apply Theorem \ref{thm:propagG}  to the set of weight $2$ newforms: $$\{f_i^\sigma\mid i\in \{1,\ldots, n\}, \sigma\in \Gal(K_i/\Q) \},$$
 whenever the orderly condition is satisfied for the finite abelian group $G$. By Theorem \ref{thm:non-van} this is the case when all the prime divisors of $|G|$ are   sufficiently large.
 Now the wanted conclusion follows directly from Corollary \ref{cor:Kato}. In the case of elliptic curves we obtain the explicit sufficient condition on the size of the prime divisors of $G$  from Theorem \ref{thm:ordrlyellipticcurves}.
\end{proof}
Restricting to $A=A_1\times A_2$, a product of two simple  abelian varieties of $\GL_2$-type, we can obtain the following sharpening.
\begin{corollary}
Let $A_1,A_2$ be simple  abelian varieties of $\GL_2$-type over $\Q$ of analytic rank $0$. Let $G$ be a finite abelian group such that for any prime $p$ dividing the order of $G$ there exists primes $\mathfrak{p}_1,\mathfrak{p}_2$ above $p$ in the coefficient fields of $A_1,A_2$, respectively, such that the Galois representations $A_1[\mathfrak{p}_1]$ and $A_2[\mathfrak{p}_2]$ are absolutely irreducible. Then there exists a constant $\kappa<1$ such that 
\begin{equation}
|\{  F\in \mathcal{F}_G(X)\mid   A_1(F),A_2(F)\text{ \rm are finite} \}| \gg \frac{X^{a(G)}}{(\log X)^{\kappa}},\quad \text{as $X\rightarrow \infty$.}
\end{equation}

\end{corollary}
\begin{proof}
 Let $A_i$ be isogenous to $A_{f_i}$ with $f_i$ a newform of weight $2$. Then the assumption on $A_i[\mathfrak{p}_i]$ translates to the fact that $\overline{\rho}_{f_i,\mathfrak{p}_i}$ is absolutely irreducible. Thus the results follows directly from Theorem \ref{thm:propagG} and Corollary \ref{cor:Kato} in view of  Proposition \ref{prop:two-non-van}. 
\end{proof}
Note that applying the above corollary to two elliptic curves  $A_1=E_1$ and $A_2=E_2$ with irreducible residual representations for all primes, we deduce Corollary \ref{cor:E1E2} from the introduction. Here we use that for elliptic curves, irreducibility alone (and not absolute irreducibility) of the mod $2$ representation implies that the orderly primes have positive density (see the proof of Proposition \ref{prop:p=2}).

\section{Appendix}\label{sec:appendix}

In this appendix we prove Proposition \ref{counter_example}. The ideas are explained in the first section, and the second section provides MAGMA code that verifies the details of the proof.

\subsection{Proof Proposition \ref{counter_example}}

Recall that Proposition \ref{counter_example} is the following statement:

\begin{proposition*}
Suppose $q\equiv 1\pmod{4}$. Then there are odd and absolutely irreducible representations $\rho_1 , \rho_2 , \rho_3 : G_{\mathbb{Q}} \rightarrow \mathrm{GL}_2(\mathbb{F}_q)$ such that for all $m$, there are no orderly primes of order $m$ for $\rho_1 \oplus \rho_2 \oplus \rho_3$.
\end{proposition*}

\noindent
The proof is by an explicit construction, and the MAGMA script in the next section verifies the claims made along the way. We are motivated by the fact that any automorphism of a biquadratic extension acts trivially on at least one of the three quadratic subfields. Experimenting with GAP and the LMFDB database, we found a degree $32$ number field $K$ that is biquadratic over a degree $8$ number field $K_0$ such that the three intermediate quadratic subfields of $K/K_0$ are Galois over $\mathbb{Q}$, and their Galois groups have odd, absolutely irreducible, faithful representations $\rho_1,\rho_2,\rho_3$ into $\mathrm{GL}_2(\mathbb{F}_q)$. Moreover, we can choose $K$ and $\rho_1,\rho_2,\rho_3$ such that \emph{every} element in the image of $\rho_1\oplus\rho_2\oplus\rho_3$ has $1$ as an eigenvalue.\\

\noindent
Our construction is closely related to the counterexample in the final section of Katz's paper \cite{katz81}. However, Katz's example is not defined over $\mathbb{Q}$, and it is a direct sum of three \emph{reducible} representations. Many related examples can also be found in the work of Cullinan \cite{cullinan07}, but they do not include the construction we present here.

\begin{proof}
Let $G$ be the group with presentation
\begin{equation*}
 \left\langle a,b,c\;\mid \; a^4 = b^4 = c^2 =1,\; ab=b a,\;c a c^{-1} = a b^2,\; cbc^{-1} = a^2 b^3 \right\rangle.
\end{equation*}
This is a finite group of order $32$, and has GAP ID $[32,33]$ \cite{smallgroup32}. The center of $G$ is the subgroup generated by $a^2$ and $b^2$; this assertion is verified by the MAGMA script in the next section. Moreover, the quotients $G/\langle a^2 \rangle$, $G /\langle b^2 \rangle$ and $G  / \langle a^2 b^2 \rangle$ are all isomorphic to the \emph{Pauli group}
\begin{equation*}
P := \left\langle u,v,w\;\mid\; u^4 = v^4 = w^2 =1,\; u^2 = v^2,\; w u w^{-1} = u^{-1},\; uv=vu,\; vw=wv \right\rangle
\end{equation*}
which is a group of order $16$ and has GAP ID $\left[ 16,13 \right]$ \cite{pauligroup}. The isomorphisms are given by
\begin{center}
\begin{tabular}{lll}
$G / \langle a^2\rangle \cong P$ & $G / \langle b^2 \rangle \cong P$ & $G / \langle a^2 b^2 \rangle \cong P$\\
  $b \mapsto u$ & $ab \mapsto u$ & $a\mapsto u$\\
  $ab\mapsto v$ & $a \mapsto v$ & $b\mapsto v$\\
  $c\mapsto w$ & $c\mapsto w$ & $c\mapsto w$\\
\end{tabular}
\end{center}
The MAGMA script in the next section verifies that these maps indeed are isomorphisms. The Pauli group has a faithful two dimensional representation $P \hookrightarrow \mathrm{GL}_2(\mathbb{F}_q )$ defined by
\begin{equation*}
u \mapsto
\begin{pmatrix}
0&i\\
i&0 \\
\end{pmatrix}, \quad v \mapsto
\begin{pmatrix}
i&0\\
0&i \\
\end{pmatrix}, \quad w\mapsto
\begin{pmatrix}
1&0\\
0&-1 \\
\end{pmatrix}
\end{equation*}
where $i \in \mathbb{F}_q$ denotes a square root of $-1$. It is easy to see that this representation is absolutely irreducible, e.g. the two eigenspaces of $u$ are spanned by $(1,1)$ and $(1,-1)$ but neither of these is an eigenspace for $w$.\\

\noindent
We now let $K$ be a number field with $\mathrm{Gal}(K/\mathbb{Q} ) \cong G$ such that $c$ is a complex conjugation in $K$. An example of such a $K$ can be found in the LMFDB database, e.g. if $K$ is the splitting field of the polynomial
\begin{equation*}
x^{16} - 60x^{12} - 288x^{10} - 666x^8 - 864x^6 - 540x^4 + 81, 
\end{equation*}
see \cite[\href{https://www.lmfdb.org/NumberField/16.4.2393397489569403764736.3}{Number Field 16.4.2393397489569403764736.3}]{lmfdb} To verify that $c$ acts as complex conjugation in $K$, it is enough to check that $K$ has no real places but the fixed field of $c$ has a real place. We implement $K$ in the MAGMA script below.

Define subfields
\begin{equation*}
K_1 := K^{\left\langle a^2 \right\rangle}, \quad K_2 := K^{\left\langle b^2 \right\rangle}, \quad K_3 := K^{\left\langle a^2 b^2 \right\rangle},
\end{equation*}
and let $\rho_i: G_{\mathbb{Q}} \rightarrow \mathrm{GL}_2(\mathbb{F}_q)$ be the representation given by
\begin{equation*}
G_{\mathbb{Q} } \twoheadrightarrow \mathrm{Gal}(K_i / \mathbb{Q}) \cong  P \hookrightarrow \mathrm{GL}_2(\mathbb{F}_q)
\end{equation*}
for $i=1,2,3$. We then set $\rho := \rho_1 \oplus \rho_2 \oplus \rho_3$, and this representation factors through $\mathrm{Gal}(K/\mathbb{Q}) \cong G$. Tracing through definitions and identifications, we see that
\begin{equation*}
\rho(a) =
\begin{pmatrix}
0&1&&&&\\
1&0&&&&\\
&&i&0&&\\
&&0&i&&\\  
&&&&0&i\\
&&&&i&0\\ 
\end{pmatrix},\;
\rho(b) =
\begin{pmatrix}
0&i&&&&\\
i&0&&&&\\
&&0&1&&\\
&&1&0&&\\
&&&&i&0\\
&&&&0&i\\
\end{pmatrix},\;
\rho(c) =
\begin{pmatrix}
1&0&&&&\\
0&-1&&&&\\
&&1&0&&\\
&&0&-1&&\\
&&&&1&0\\
&&&&0&-1\\
\end{pmatrix}
\end{equation*}
and one can then verify directly that every element in the image of $\rho$ has $1$ as an eigenvalue. We refer the MAGMA script in the next section for this computation. Moreover, since $c$ acts a complex conjugation, it is clear from the expression for $\rho(c)$ that $\det\rho_i(c)=-1$ for all $i=1,2,3$ so $\rho_1$, $\rho_2$ and $\rho_3$ are odd.
\end{proof}

\subsection{MAGMA-code}

The following MAGMA code can be used to verify the claims made in the proof of Proposition \ref{counter_example}. The code can be executed on the online MAGMA calculator available here: \url{https://magma.maths.usyd.edu.au/calc/}.\\

\noindent
We first implement $G$ as the Galois group of $K$ and verify that its GAP ID is [32,33].\\

\noindent
\begin{verbatim}
R<x> := PolynomialRing(RationalField());
f := x^16-60*x^12-288*x^10-666*x^8-864*x^6-540*x^4+81;
K := SplittingField(f);
G := AutomorphismGroup(K);
print"G has GAP ID [32,33]";
IdentifyGroup(G) eq <32,33>;

\end{verbatim}

\noindent
We then find generators $a$, $b$ and $c$ of $G$ that give the presentation for $G$ in the proof of Proposition \ref{counter_example}.\\

\begin{verbatim}

a := G.1;
b := G.3*G.4;
c := G.1*G.4;
print"a,b,c generate G and satisfies the relations";
G eq sub<G | a, b, c> and
a^4 eq Id(G) and
b^4 eq Id(G) and
c^2 eq Id(G) and
a*b eq b*a and
c*a*c^-1 eq a*b^2 and
c*b*c^-1 eq a^2*b^3;  

\end{verbatim}

\noindent
We then verify that the center of $G$ is equal to $\langle a^2,b^2\rangle$. 

\begin{verbatim}
	
H := sub<G | a^2, b^2>;
print"Z(G) = <a^2,b^2>";
Centre(G) eq H;

\end{verbatim}

\noindent
Then we implement the Pauli group and verify that it has GAP ID $[16,13]$.

\begin{verbatim}
	
F<u,v,w> := FreeGroup(3);
rels := {u^4=Id(F), v^4=Id(F), w^2=Id(F), u^2=v^2,
w*u*w^-1=u^-1, u*v=v*u, v*w=w*v};
P<u,v,w> := quo<F | rels>;
print"P has GAP ID [16,13]";
IdentifyGroup(P) eq <16,13>;

\end{verbatim}

\noindent
Next, we verify the isomorphisms from the quotients to the Pauli group.

\begin{verbatim}
	
H1 := sub<G | a^2>;

u := b;
v := a*b;
w := c;

print"G/<a^2> = P";
u^4 in H1 and v^4 in H1 and w^2 in H1 and 
u^2*v^-2 in H1 and w*u*w^-1*u in H1 and 
u*v*u^-1*v^-1 in H1 and v*w*v^-1*w^-1 in H1;

print"";

H2 := sub<G | b^2>;

u := a*b;
v := a;
w := c;

print"G/<b^2> = P";
u^4 in H2 and 
v^4 in H2 and
w^2 in H2 and
u^2*v^-2 in H2 and
w*u*w^-1*u in H2 and 
u*v*u^-1*v^-1 in H2 and
v*w*v^-1*w^-1 in H2;

print"";

H3 := sub<G | a^2*b^2>;

u := a;
v := b;
w := c;

print"G/<a^2*b^2> = P";
u^4 in H3 and
v^4 in H3 and
w^2 in H3 and 
u^2*v^-2 in H3 and
w*u*w^-1*u in H3 and
u*v*u^-1*v^-1 in H3 and
v*w*v^-1*w^-1 in H3;

\end{verbatim}

\noindent
And then we verify the matrix representation of $P$.

\begin{verbatim}

print"Verify the matrix representation of P";
R<x> := PolynomialRing(RationalField());
L<i> := NumberField(x^2 + 1);
M2 := KMatrixSpace(L,2,2);
u := M2 ! [[0,i],[i,0]];
v := M2 ! [[i,0],[0,i]];
w := M2 ! [[1,0],[0,-1]];
P_matrix := sub< GL(2, L) | u, v, w >;
IdentifyGroup(P_matrix) eq <16,13> and
u^4 eq Id(P_matrix) and
v^4 eq Id(P_matrix) and
w^2 eq Id(P_matrix) and
u^2 eq v^2 and 
w*u*w^-1 eq u^-1 and 
u*v eq v*u and
v*w eq w*v;

\end{verbatim}

\noindent
We can now verify that every element in the image of $\rho$ has $1$ as an eigenvalue.

\begin{verbatim}
	
M6 := KMatrixSpace(L,6,6);
rho_a := M6 ! [[0,1,0,0,0,0], [1,0,0,0,0,0], [0,0,i,0,0,0],
[0,0,0,i,0,0], [0,0,0,0,0,i], [0,0,0,0,i,0]];
rho_b := M6 ! [[0,i,0,0,0,0], [i,0,0,0,0,0], [0,0,0,1,0,0],
[0,0,1,0,0,0], [0,0,0,0,i,0], [0,0,0,0,0,i]];
rho_c := M6 ! [[1,0,0,0,0,0], [0,-1,0,0,0,0], [0,0,1,0,0,0],
[0,0,0,-1,0,0], [0,0,0,0,1,0], [0,0,0,0,0,-1]];
G_matrix := sub< GL(6, L) | rho_a, rho_b, rho_c >;
print"Every element in the image of rho has 1 as an eigenvalue";
eigenvalue_one := true;
for g in G_matrix do 
    P := CharacteristicPolynomial(g);
    if Evaluate(P,1) ne 0 then
        eigenvalue_one = false;
        break;
    end if;
end for;
print eigenvalue_one;

\end{verbatim}

\noindent
Finally, we verify that $c$ acts as complex conjugation.\\

\noindent
\begin{verbatim}

print"c acts as complex conjugation";
E := FixedField(K, sub<G | c>);
RealPlaces(K) eq [] and RealPlaces(E) ne [];    
\end{verbatim}

\bibliography{references}{} \bibliographystyle{plain}
\end{document}